\documentclass[psamsfonts]{amsart}

\usepackage{amssymb,amsfonts}
\usepackage[all,arc]{xy}
\usepackage{enumerate}
\usepackage[margin=1.2in]{geometry}
\usepackage{mathrsfs}
\usepackage{hyperref}
\hypersetup{
pdftitle={},
pdfsubject={},
pdfauthor={Kolosov Petro},
pdfkeywords={}
}

\usepackage[usenames,dvipsnames]{xcolor}

\usepackage{bm}
\usepackage{pgfplots}
\usepackage{verbatim}

\numberwithin{equation}{section}
\newtheorem{thm}{Theorem}[section]
\newtheorem{cor}[thm]{Corollary}
\newtheorem{prop}[thm]{Proposition}
\newtheorem{lem}[thm]{Lemma}
\newtheorem{conj}[thm]{Conjecture}
\newtheorem{quest}[thm]{Question}

\theoremstyle{definition}
\newtheorem{defn}[thm]{Definition}

\theoremstyle{remark}
\newtheorem{rem}[thm]{Remark}

\newcommand{\note}[1]{\textbf{\textcolor{blue}{#1}} \\ \\}

\DeclareMathOperator{\Ric}{Ric}

\DeclareMathOperator{\arcsinh}{arcsinh}
\DeclareMathOperator{\Hess}{Hess}
\DeclareMathOperator{\tn}{tn}
\DeclareMathOperator{\cs}{cs}
\DeclareMathOperator{\sn}{sn}

\title[Spectral Gap Estimates on conformally flat Manifolds ]{Spectral Gap Estimates on Conformally Flat Manifolds}
\author{Gabriel Khan}\address[Gabriel Khan]{Department of Mathematics, Iowa State University, Ames, IA, USA.}
  \email{gkhan@iastate.edu}  
\author{Malik Tuerkoen}
\address[Malik Tuerkoen]{Department of Mathematics, University of California,  Santa Barbara, CA, USA.}
  \email{mmtuerkoen@ucsb.edu}
\begin{document}
\maketitle
\begin{abstract}
The fundamental gap is the difference between the first two Dirichlet eigenvalues of a Schr\"odinger operator (and the Laplacian, in particular). For horoconvex domains in hyperbolic space, Nguyen, Stancu and Wei conjectured that it is possible to obtain a lower bound on the fundamental gap in terms of the diameter of the domain and the dimension \cite{nguyen2022fundamental}. In this article, we prove this conjecture by establishing conformal log-concavity estimates for the first eigenfunction. This builds off earlier work by the authors and Saha \cite{khan2024concavity} as well as recent work by Cho, Wei and Yang \cite{cho2023probabilistic}. We also prove spectral gap estimates for a more general class of problems on conformally flat manifolds and investigate the relationship between the gap and the inradius. For example, we establish gap estimates for domains in $\mathbb{S}^1 \times \mathbb{S}^{N-1}$ which are convex with respect to the universal affine cover.
\end{abstract}

\section{Introduction}

The goal of this paper is to establish lower bounds on the fundamental gap of sufficiently convex domains in conformally flat manifolds, with a particular focus on horoconvex domains in hyperbolic space. Given a bounded domain $\Omega$ in a Riemannian manifold $(M,g)$, we consider the Dirichlet eigenvalue problem
\begin{eqnarray}\label{Dirichlet-Laplacian}
    -\Delta _g u = \lambda(\Omega) u \quad\textup{in }\Omega, \quad \quad u = 0 \quad\textup{on }\partial \Omega,
\end{eqnarray}
where $\Delta_g$ is the Laplace operator.
There is a discrete sequences of eigenvalues 
$$0<\lambda_1(\Omega) \leq \lambda_2(\Omega)\leq \cdots \leq \lambda_k(\Omega)\leq \cdots\rightarrow +\infty,$$
repeated with multiplicity.
The \emph{fundamental gap} is the difference between the first and second eigenvalues
\begin{align*}
    \Gamma(\Omega) = \lambda_2(\Omega)-\lambda_1(\Omega)>0,
\end{align*}
which is non-zero whenever the domain is connected. For convex domains in $\mathbb R^N,$
Andrews and Clutterbuck \cite{andrews2011proof} proved that the fundamental gap of the Schr\"odinger operator $-\Delta + V$ is greater than the Dirichlet gap of a one-dimensional model: 
\begin{align*}
    \Gamma(\Omega, V) \geq \overline\lambda_2(\overline V) - \overline \lambda_1(\overline V).
\end{align*}
Here, $\overline \lambda_i$ is the $i$-th eigenvalue of the operator $\mathcal L = -\tfrac{d^2}{ds^2}+\overline V$ on $[-D/2,D/2],$ and $\overline V'$ is a \emph{modulus of convexity} for $V$. In other words, for any $x,y\in \Omega$, we have that 
\begin{align*}
    \langle \nabla V(y), \gamma_{x,y}'(\tfrac {d(x,y)}{2})\rangle -  \langle \nabla V(x), \gamma_{x,y}'(\tfrac {-d(x,y)}{2})\rangle \geq 2 \overline V'(\tfrac{d(x,y)}{2}), 
\end{align*}
 where $\gamma_{x,y}$ is the minimizing geodesic such that $\gamma_{x,y}(\tfrac{-d(x,y)}{2}) = x$ and $\gamma_{x,y}(\tfrac{d(x,y)}{2}) = y.$

There has been a large amount of work studying the fundamental gap in other geometries as well.
For convex regions on the sphere $\mathbb S^N$, Dai, He, Seto, Wang and Wei (in various subsets) proved the analogous result to the Andrews-Clutterbuck theorem
\cite{10.4310/jdg/1559786428, He-Wei2020, Dai-Seto-Wei2021} (see also \cite{lee1987estimate, wang2000estimation} for earlier work). Recently the authors along with Nguyen, Saha and Wei (in various subsets) made progress in obtaining spectral gap estimates for deformations of round spheres in dimension two \cite{surfacepaper1,khan2023modulus} and for conformal deformations of the sphere in all dimensions \cite{khan2024concavity}.

 For spaces of negative curvature, the behavior of the fundamental gap of convex domains can be remarkably different compared to flat or spherical geometry. In fact, Bourni et al. \cite{bourni2022vanishing} showed that for any $D>0$ there are convex domains $\Omega$ of any diameter $D$ in hyperbolic $N$-space $\mathbb{H}^N$   whose spectral gap is arbitrarily small. The first-named author and Nguyen then extended this result to spaces of mixed sectional curvature (for small diameter) and negatively pinched curvature for arbitrary diameter \cite{khan2022negative}.

Since assuming convexity is not enough to ensure uniform lower bounds on the fundamental gap in such geometries, it is natural to consider domains which satisfy stronger convexity assumptions. In particular, Nguyen, Stancu and Wei  \cite{nguyen2022fundamental} conjectured that there is a lower bound on the fundamental gap of domains which are \emph{horoconvex} (see Definition \ref{Def: horoconvexity} for the definition of horoconvexity).

\begin{conj}\label{conjecture: horoconvex-conjecture}
    There exists a constant $c(N,D)>0,$ such that for any horoconvex domain $\Omega \subset \mathbb H^N$, the following inequality holds:
    \begin{align*}
        \Gamma(\Omega) \geq c(N,D).
    \end{align*}
\end{conj}
In recent work \cite{khan2024concavity}, the authors along with Saha showed that for $N  = 2$ and $D \leq 2  \textup{arccsch}(2 \sqrt{11/3} )$, the gap satisfies $\Gamma(\Omega) \geq\tfrac{32}{3 (7 + \sqrt{33})} \tfrac{\pi^2}{ D^2} + \tfrac{4}{3}.$ Therefore, the conjecture holds in two-dimensions when the diameter is not too large.
    
In this article, we drop the restrictions on the the dimension and the diameter in order to prove the general conjecture. 
\begin{thm} \label{The horoconvexity theorem}
    Conjecture \ref{conjecture: horoconvex-conjecture} holds.
\end{thm} 

Apart from this question, we are able to obtain fundamental gap estimates for domains in conformally flat manifolds which are convex with respect to a flat connection. 
\begin{thm}\label{thm: Gap-Estimate-for-conformally-flat-manifolds}
    Suppose that $\Omega \subset (M, g = e^{2\varphi}g_{\textup{E}})$ is a convex domain with respect to Euclidean geometry, where $g_{\textup{E}}$ denotes the Euclidean metric. Then the fundamental gap of $\Omega$ satisfies \begin{align}
        \Gamma(\Omega, \Delta_g) \geq \frac{\min \overline \rho}{\max_\Omega \exp(2\varphi)}\overline \Gamma(\overline \rho, \overline V)-\frac{N-2}{4(N-1)}\left(\max_\Omega R_g -\min_\Omega R_g\right),
    \end{align}
    where $R_g$ denotes the scalar curvature and $\overline \rho$ and $\overline V$ which can be solved explicitly in terms of $\varphi$ and $\lambda_1(\Omega)$.
\end{thm}

It is worthwhile to compare this result to those Oden-Song-Wang and Ramos et al. \cite{oden1999spectral,ramos2023integral}, which establish lower bounds on the fundamental gap for domains which satisfy an interior rolling ball condition. However, this is a comparatively restrictive class of domains and requires that the boundary is uniformly $C^2$. On the other hand, our estimates require conformal flatness but only depend on the diameter, conformal factor (and possibly the inradius).

The strategy to prove these results is to study the eigenfunctions of the problem \eqref{Dirichlet-Laplacian} with respect to Euclidean metric, using a conformal deformation. More generally, if $\tilde g = e^{2\varphi}g$ and $\tilde u$ is a $\tilde g$ Dirichlet eigenfunction, then $u = e^{\tfrac{N-2}{2}\varphi}\tilde  u$ solves a Schr\"odinger equation of the form  \begin{eqnarray}\label{eqn: Weighted eigenfunction equation}
 -\Delta_g u + V u = \lambda(\Omega, \rho, V) \rho u \quad\textup{in }\Omega, \qquad u = 0 \quad\textup{on }\partial \Omega,
\end{eqnarray}
where $V$ is a potential and $\rho$ is a weighting function (see Section \ref{Background Section} for details). Therefore, we can establish estimates on the gap of domains in conformally flat geometry by studying them in terms of an underlying conformal connection (which will typically be chosen to have constant or vanishing curvature).

To obtain spectral gap estimates, we compare the gap of the problem of \eqref{eqn: Weighted eigenfunction equation} to the one-dimensional problem:  
\begin{align}\label{eqn: one-dimensional-model-weighted}
    -\varphi ''+ \overline V \varphi= \overline\lambda(\overline \rho, \overline V)\overline \rho \varphi\quad \textup{in }[-\tfrac D2, \tfrac D2],   \quad \varphi(-\tfrac D2) = \varphi( \tfrac D2) = 0
\end{align}
where $\overline{V}'$ is a modulus of convexity for $V$ and $\overline \rho'$ is a modulus of concavity of $\rho$. 
 More precisely, the main theorem which provides the gap estimates is the following.
\begin{thm}\label{main-thm}
    Let $\Omega \subset (M,g_{\mathbb M^N_K})$ be a convex domain, where $g_{\mathbb M^N_K}$ denotes the metric of the space of constant sectional curvature $K \geq 0$. Suppose that $\overline \rho' $ is a modulus of concavity for $\rho$ and that $\overline V$ is a modulus of convexity of $V.$ Moreover, assume the additional inequalities
    \begin{align}\label{condition-2}
        [\overline \rho(s) \cs_K^2(s)]' \left(\overline \lambda(\overline V, \overline \rho)-\lambda_1(\Omega, \rho, V)\right)\geq 0,  
    \end{align}
     where $\cs_K(s) = \cos(\sqrt K s)$ and (if $K>0$) that
    \begin{align}\label{Additional assumption for K>0}
        \overline \rho \leq \min \rho, \quad \overline V \geq \max V.
    \end{align}
    Then the spectral gap of the problem \eqref{eqn: Weighted eigenfunction equation} satisfies
    \begin{align}\label{ineq: main-inequality}
        \Gamma(\Omega, \rho, V) \geq \frac{\min \overline \rho}{\max_\Omega \rho}\overline \Gamma(\overline \rho, \overline V),
    \end{align}
    where $\overline \Gamma(\overline \rho, \overline V)$ is the spectral gap of \eqref{eqn: one-dimensional-model-weighted}.
\end{thm}

This result generalizes previous results from \cite{andrews2011proof,10.4310/jdg/1559786428}, which considered the case $\rho\equiv 1$. A major step of the proof is to establish a modulus of concavity for the log of the principle eigenfunction of \eqref{eqn: Weighted eigenfunction equation}. In essence, this estimate shows that the principle Dirichlet eigenfunction is at least as log-concave as the solution to the one-dimensional model.  
 \begin{prop}\label{thm: super-log-concavity-general}
    Under the assumptions of Theorem \ref{main-thm}, we have that for $ v = \log u_1$ that for  $x,y \in \Omega$
    \begin{align}\label{ineq: integral-over-hessian}
        \langle \nabla v(y), \gamma_{x,y}'(\tfrac{d(x,y)}{2})\rangle -\langle \nabla v(x), \gamma_{x,y}'(\tfrac{-d(x,y)}{2})\rangle\leq 2(\log \overline\varphi_1)'(\tfrac{ d(x,y)}{2})+(N-1)\textup{tn}_K(\tfrac{d(x,y)}{2})
    \end{align}
   where $\overline \varphi_1$ is the first eigenfunction of the problem \ref{eqn: one-dimensional-model-weighted} and $\textup{tn}_K$ is defined as $ \textup{tn}_K(s)=\sqrt{K} \tan (\sqrt{K}s)$.
\end{prop}
It is well-known that log-concavity estimates for the principle eigenfunction imply lower bounds on the fundamental gap (see, e.g., \cite{singer1985estimate,andrews2011proof}). 
Dividing \eqref{ineq: integral-over-hessian} by $d(x,y)$ and passing to the limit as $d(x,y) \rightarrow 0^+$, %it is possible to compare the spectral gap of \eqref{eqn: Weighted eigenfunction equation} to the gap of \eqref{eqn: one-dimensional-model-weighted} or 
gives $\Hess \log u_1 \leq (\log \overline \varphi_1)''(0)+\tfrac{(N-1)K}{2}.$ Using this upper bound one can derive a lower bound of $\Gamma(\Omega)$ using known Neumann eigenvalue comparisons \cite{charalambous2015eigenvalue} (see Section \ref{section: rough-bounds-for-horoconvex}).
In the literature, there are currently two methods to prove \eqref{ineq: integral-over-hessian}. One can either use a maximum principle argument (see e.g. \cite{andrews2011proof,10.4310/jdg/1559786428,Ni2013,khan2023modulus}) or a stochastic analysis argument using coupled diffusions. In this paper, we use the stochastic method, but with a slight variation it is possible to obtain the same result using the maximum principle (see Subsubsection \ref{Sketch of Maximum Principle Proof} for a brief discussion of this). The diffusion approach was introduced in \cite{Gong-Li-Luo2016}, where they gave an alternative proof of the fundamental gap conjecture. Recently, Cho, Wei and Yang \cite{cho2023probabilistic} generalized this technique to spherical geometry, and our work follows their approach closely.

The assumptions \eqref{condition-2} and \eqref{Additional assumption for K>0} in Theorem \ref{main-thm} appear to be technical. However, Condition \eqref{condition-2} is necessary in the following sense: in hyperbolic space, there are domains which are convex with respect to the flat connection with very small inradius (i.e., large eigenvalue) but whose spectral gap is arbitrarily small (see Section \ref{But why is the gap small for small diameter?} and the Appendix \ref{We need the inradius} for details). Therefore, it is necessary for our estimates to decay as the eigenvalue becomes large. In fact, we show that if the inradius is bounded below, the spectral gap is bounded from below by a positive constant (see Corollary \ref{Estimate-for-inradius-bounded-below}).

In the case that $\rho$ is concave, 
%Despite this restriction, there is one important case where this condition can be satisfied without making any additional assumptions on the eigenvalue of the model, which is when $\rho$ is concave. In this situation, 
we can take $\overline{\rho}$ to be a constant (or even decreasing), in which case \eqref{condition-2} is not a major obstacle and the functions $\overline \rho$ and $\overline V$ can be chosen to only depend on the diameter $D$ and the conformal factor $\exp(2\varphi)$ (see e.g. \eqref{Estimate-for-rho-concave}). When the underlying geometry is hyperbolic space (or $\mathbb{S}^1 \times \mathbb{S}^{N-1}$), the weighting function $\rho$ will not be concave, which will affect the estimates that can be established and force $\overline V$ (and thus the gap estimate) to depend on the first eigenvalue $\lambda_1(\Omega).$

\subsection{Overview of the Paper}

In Section \ref{Background Section}, we briefly review some background on hyperbolic and conformal geometry that will be used throughout the paper. In particular, we show that one can bound fundamental gaps a conformally flat metric in terms of the spectral gap of a related problem in Euclidean metric (see Proposition \ref{prop: conformally-flat-gap-scalarcurvature}). In Subsection \ref{subsection two-point-derivatives}, we consider the two-point function $Z(x,y)$ defined by (with respect to the Levi-Civita connection of $g_{\mathbb M^N_K}$)
\begin{align}\label{defn: of Z}
    Z(x,y) = \langle \nabla v(y) , \gamma_{x,y}'(\tfrac{d(x,y)}{2})\rangle - \langle \nabla v(x) , \gamma_{x,y}'(\tfrac{-d(x,y)}{2})\rangle,
\end{align}
where $\gamma_{x,y}$ denotes the geodesic such that $\gamma_{x,y}(\tfrac{-d(x,y)}{2}) = x$ and $\gamma_{x,y}(\tfrac{d(x,y)}{2}) = y$ and compute its derivatives for $ v = \log u_1,$ where $u_1$ denotes the first eigenfunction of the problem \eqref{eqn: Weighted eigenfunction equation}. The computation is similar to the one presented in \cite{10.4310/jdg/1559786428, khan2023modulus} but here we calculate this at any point $(x,y)$, not necessarily a maximal point (see also \cite{cho2023probabilistic}).

In Section \ref{Stochastic Analysis on Manifolds}, we first review some basic facts about stochastic analysis on manifolds. We then prove Theorem \ref{main-thm} via a coupled diffusion method, following \cite{cho2023probabilistic}. %From this, we are able to prove Theorem \ref{main-thm}. More precisely, we define a diffusion on $(X_t,Y_t)$ \eqref{SDE-3} (inspired from the computations of Section \ref{subsection two-point-derivatives}) and let $F_t =Z(X_t,Y_t) $ and $\xi_t = d(X_t,Y_t)/2$ and compute their SDEs  $F_t$ (see Proposition \ref{prop: SDE for F} and Proposition \ref{prop: Itos formula for distance}). With these at hand, we then prove. 
  In Section \ref{One-dimensional-comparison}, we study the eigenvalues of the model \eqref{eqn: one-dimensional-model-weighted} and prove a lower bound on the fundamental gap in dimension one, which we will then use for the comparison of the gap. In Section \ref{Section: horoconvex-estimates}, we turn our attention to horoconvex domains and prove Theorem \ref{The horoconvexity theorem}. We then derive a more explicit lower bound of the spectral gap for horoconvex domains using a comparison of Neumann eigenvalues. More precisely, we consider $w= \frac{u_2}{u_1}$, which satisfies
\begin{equation} \label{The gap is the eigenvalue}
       \Delta w +2\nabla \log u_1\cdot \nabla w=-\Gamma(\Omega, \rho)\rho w
    \end{equation}
    with Neumann boundary conditions \cite{singer1985estimate}. Then from Proposition  \ref{thm: super-log-concavity-general} one obtains an upper bound on the Hessian of $ v = \log u$ which then induces a lower bound on the Bakry-\'Emery-Ricci curvature \cite{charalambous2015eigenvalue}.
    
 In Section \ref{Section: general-conformally-flat}, we prove Theorem \ref{thm: Gap-Estimate-for-conformally-flat-manifolds}. We also investigate the relationship between the inner radius of a given domain $\Omega$ and its spectral gap. In particular, we show that if the conformal factor is concave with respect to the Euclidean metric $g_{\textup{E}}$ , one can prove lower bounds for $g_{\textup{E}}$-convex domains that are independent of the inner radius. However, in general it is only possible to prove lower bounds depending on the diameter if we restrict ourselves to domains which are $g_{\textup{E}}$ convex and whose inradius is controlled by the diameter (see Section \ref{Section: general-conformally-flat} for more details). We then prove a lower bound on the fundamental gap for small conformal deformation of the round sphere (see Theorem \ref{thm: small-deformation-of-sphere}).

\subsection*{Acknowledgements}

Malik T. wishes to thank Guofang Wei for her helpful comments and guidance.  G. Khan would like to thank David Herzog for some helpful conversations. He is supported in part by Simons Collaboration Grant 849022.

\section{Background and Preliminaries}\label{Background Section}
%\subsubsection*{Notations} For two real numbers $a,b$ we denote $a\wedge b: = \min\{a,b\}$ and $a\vee b = \max\{a,b\}.$

\subsection{Horospheres and horoconvexity}

Since Theorem \ref{The horoconvexity theorem} concerns horoconvex domains in hyperbolic geometry, we review some basic properties of such domains.

In hyperbolic space, a horosphere is a specific type of hypersurface which can be formed as the limit of expanding balls in hyperbolic space which share a tangent hyperplane and its point of tangency. Equivalently, one can consider them as spheres whose center lies on the ideal boundary of $\mathbb{H}^N$. In the Poincar\'e disk model of hyperbolic space, horospheres are Euclidean spheres inside the unit ball $B$ which are tangent to $\partial B.$

\begin{defn}
\label{Def: horoconvexity}
We say that a domain $\Omega \subset \mathbb H^N$ is horoconvex if at every point $p \in \partial \Omega$ there exists a horosphere passing through $p$ which contains $\Omega$ entirely.
\end{defn}

Equivalently, a domain $\Omega$ is horoconvex if and only if all the principal
curvatures of its boundary hypersurface are at least one. In this way, horoconvexity is a strengthening of convexity.

\subsection{Conformal deformations of the Laplace operator}

We now review some properties of conformal geometry. Given a function $\varphi: M\rightarrow \mathbb R,$ we consider the conformal metric  $\tilde g = e^{2\varphi}g.$ There are known formulas the Hessian and the Laplace operator of under conformal change, which states that for a smooth function $F:M \rightarrow \mathbb R$,
\begin{align}
\textup{Hess}_{\tilde g} F &= \textup{Hess}_g F - 2 d\varphi \otimes d F +(\nabla \varphi \cdot \nabla F) g \label{eqn: conformal change of Hessian} \\
\label{eqn: conformal change laplacian formula}
 \Delta _{\tilde g}F &= e^{-2\varphi}\Bigl(\Delta_g F +(N-2)\nabla \varphi\cdot \nabla F\Bigr).
\end{align}
 Moreover, the scalar curvature $R_g$ of the conformal metric satisfies: 
\begin{align}\label{eqn: Scalar-curvature-under-conformal-deformation}
   R_{e^{2\varphi}g}  = e^{-2\varphi}\left(R_g-2(N-1) \Delta_g \varphi-(N-1)(N-2)|\nabla \varphi|^2\right). 
\end{align}

Now, we consider a  $\tilde g$-eigenfunction $\psi$ satisfying
\begin{align}\label{conformal eigenfunction equation}
    \Delta_{\tilde g} \psi = -\lambda \psi  \quad \textup{in }\Omega \quad \textup{and }\quad  \psi = 0 \quad \textup{on }  \partial \Omega.
\end{align}
A straightforward argument shows the function $u = \psi e^{\tfrac{N-2}{2}\varphi}$ is a weighted eigenfunction of the following Schr\"odinger operator of $g$.
\begin{equation}\label{eqn: Conformal eigenfunction equation, Schrodinger form}
      \begin{cases}
       -\Delta_g u +\Bigl[ \tfrac{(N-2)^2}{4}  |\nabla \varphi |^2+\tfrac{N-2}{2}  \Delta_g \varphi \Bigr]u = \lambda   e^{2\varphi}  u\\
      u \vert_{\partial \Omega} = 0
   \end{cases},
\end{equation}
which is a weighted Schr\"odinger equation with weight $\displaystyle \rho =e^{2\varphi}$ and potential \begin{align}\label{eqn: V and scalar-curvature}
    V =  \tfrac{(N-2)^2}{4}  |\nabla \varphi |^2+\tfrac{N-2}{2}  \Delta_g \varphi= - \frac{N-2}{4(N-1)}e^{2\varphi}R_{\tilde g}+\frac{N-2}{4(N-1)}R_g.
\end{align}

In this paper, we study equations of the type \eqref{eqn: Weighted eigenfunction equation} and so we choose the conformal factor to facilitate this analysis. We denote the eigenvalues by $\lambda_i(\Omega,\Delta_g, \rho, V).$ In case $V \equiv 0,$ we just write $\lambda_ i(\Omega, \Delta_g, \rho).$ If clear from the context, we omit the dependence of $\Delta_g.$ Moreover, for any function $f :\Omega \rightarrow \mathbb R,$ we denote $\textup{osc}_\Omega (f) = \sup_\Omega f -\inf_\Omega f$ to be its oscillation.
\begin{prop}\label{prop: conformally-flat-gap-scalarcurvature}
Let $\Omega \subset(M,g= e^{2\varphi}g_{\textup{E}}).$ Then 
\begin{align*}
    \Gamma(\Omega, \Delta_{g})  \geq \Gamma(\Omega, \Delta_{g_{\textup{E}}}, e^{2\varphi})-\frac{N-2}{4(N-1)}\textup{osc}_\Omega (R_g),
\end{align*} 
where $R_g$ denotes the scalar curvature of $g.$
\end{prop}
\begin{proof}
    In view of \eqref{eqn: Conformal eigenfunction equation, Schrodinger form}, we only need to study the eigenvalues of the problem  \begin{align*}
    -\Delta \tilde u +V\tilde u = \lambda e^{2\varphi}\tilde u \quad \textup{in } \Omega,\quad \tilde u = 0 \quad \textup{on }\partial \Omega.
\end{align*}
Note that 
\begin{align*}
    V = -\frac{N-2}{4(N-1)}e^{2\varphi}R_g,
\end{align*}
where $R_g$ is the scalar curvature of the metric $g.$ 
Note that for $\rho = \exp (2\varphi)$, the eigenvalues satisfy 
\begin{align*}
    \lambda_1(\Omega, g) &= \inf \frac{\int_\Omega |\nabla f|^2+Vf^2\, dx}{\int _\Omega \rho f^2\, dx}\\
    &= \inf \frac{\int_\Omega |\nabla f|^2-\tfrac{N-2}{4(N-1)}\rho R_g f^2\, dx}{\int _\Omega \rho f^2\, dx}.
\end{align*}
Thus, 
\begin{align}\label{Boundedness-of-evalues}
    \lambda_i(\Omega, \rho)-\frac{N-2}{4(N-1)}\max_\Omega R_g\leq\lambda_i(\Omega, g)\leq \lambda_i(\Omega, \rho) -\frac{N-2}{4(N-1)}\inf_\Omega R_g.
\end{align}
In particular, we have that for the fundamental gap, 
\begin{align*}
    \Gamma(\Omega, g) \geq \Gamma(\Omega,\Delta_{g_{\textup{E}}}, \rho)-\frac{N-2}{4(N-1)}\textup{osc}_\Omega (R_g).
\end{align*}
\end{proof}
\begin{rem}
    The proof above shows that if $(M,g)$ has constant scalar curvature, then we have that 
    \begin{align}\label{Eigenvalue-under-conformal-change-explicit-scalarcurvature}
        \lambda_i(\Omega, \Delta_g) = \lambda_i(\Omega, \Delta_{
        g_{\textup{E}}}, e^{2\varphi})-\frac{(n-2)R_g}{4(n-1)}
    \end{align}
\end{rem}

From Proposition \ref{prop: conformally-flat-gap-scalarcurvature}, in the case of constant scalar curvature, it is possible to bound the fundamental gap if we can obtain a bound for the problem
\begin{align}\label{eqn: only-weighted-no-potential}
    -\Delta_{g_{\textup{E}}} u = e^{2\varphi} \lambda u \quad \textup{in }\Omega \quad u = 0 \quad \textup{on }\partial \Omega.
\end{align}

\subsection{Derivatives of Two-Point Functions in Curved Geometries}\label{subsection two-point-derivatives}
In this section, we compute the second-derivatives of a two-point function $Z$ in a Riemannian manifold. In previous works \cite{10.4310/jdg/1559786428, khan2023modulus}, these derivatives were computed explicitly at the points which maximize the function. In this paper, we calculate the derivative without assuming that we are at a maximum, since this will be needed for the probabilistic argument to obtain a modulus of concavity.

Let $\Omega \subset (M, g_{\mathbb M^N_K})$ be uniformly convex. In other words, for any two different points $x,y\in \Omega$, there is a unique minimizing geodesic which is contained entirely in $\Omega$. We denote this geodesic by $\gamma_{x,y}$ and use the convention that $\gamma_{x,y}(\tfrac{-d(x,y)}{2}) = x$ and $\gamma_{x,y}(\tfrac{d(x,y)}{2}) = y.$ We then let $v = \log u_1,$ where $u_1$ is the first Dirichlet eigenfuntion of the problem \eqref{eqn: Weighted eigenfunction equation} and consider the function $Z:\Omega\times \Omega \rightarrow \mathbb R$ defined by \eqref{defn: of Z}
%\begin{align*}
%    Z(x,y) = \langle \nabla v(y), \gamma_{x,y}'(\tfrac{d}{2})\rangle - \langle \nabla v(x), \gamma_{x,y}'(\tfrac{-d}{2})\rangle, 
%\end{align*}
%where here and in the following we denote $d = d(x,y).$

 Now, we let $x,y$ be a distinct pair of points in $\Omega$ and denote the set of such points as $\textrm{Conf}_2(\Omega)$ (short for configuration space).
% \[(x,y) \in \hat \Omega := \{(x,y) \in \Omega \times \Omega \, |\, x \neq y\}.\] 
  In order to compute derivatives of $Z$ at $(x,y)$, we first define a frame along $\gamma_{x,y}$, which will be helpful in the computations. More precisely, we let $\{e_i\}_{1\leq i\leq N}$
 be an orthonormal basis at $x$ with $e_n=\gamma'_{x,y}(\tfrac{-d}{2})$ and parallel transport $\{e_i\}_{1\leq i\leq N}$ along $\gamma_{x,y}$. We then define the vectors
\[ E_i= e_i\oplus e_i,\ 1\leq i\leq N-1,\quad \textup{and } \quad E_N=e_N\oplus( -e_N).  \] 
For notational purposes, we introduce the following notation. For a vector field $X$ on $\Omega,$ we define the two-point function $\mathcal F_X(x,y)$ as follows 
\begin{align}\label{defn: notation for mathcal D}
    \mathcal F_X(x,y) = \langle X(y), \gamma_{x,y}'(\tfrac d2)\rangle - \langle X(x), \gamma_{x,y}'(\tfrac {-d}{2})\rangle .
\end{align}
Recall that $\overline \rho$ is said to be a \textit{modulus of concavity of $\rho$} if 
\begin{align*}
    \mathcal F_{\nabla \rho}(x,y) \leq 2 \overline \rho'(\tfrac d2)
\end{align*}
and we call $\overline V$ a \textit{modulus of convexity} of $V$ if $-\overline V$ is a modulus of concavity of $-V,$ i.e. if 
\begin{align*}
    \mathcal F_{\nabla V}(x,y) \geq 2 \overline V'(\tfrac d2).
\end{align*}
Finally, given a function $f$ on $\Omega,$ we define the two point function $\mathcal C_f(x,y)$ as follows \begin{align}\label{defn of mathcal C}
    \mathcal C_f(x,y) = f(x) + f(y).
\end{align}

\begin{prop}\label{Second-order-derivative}
 Using the above notation, one has that 
 \begin{align*}
     \sum_{i =1 }^N \nabla ^2_{E_i,E_i}Z(x,y) &= -2 \tn_K(\tfrac d2)\nabla _{E_N}Z(x,y) -2\nabla_{\nabla v(x) \oplus \nabla v(y)}Z(x,y)+(N-1)[K-\tn_K^2(\tfrac d2)] Z(x,y) \\
     &\quad-2\tn_K(\tfrac d2)\mathcal C_{[V-\lambda_1\rho]}(x,y)+\mathcal F_{\nabla [V-\lambda_1 \rho]}(x,y)    \\
     &\quad +2\tn_K(\tfrac d2)\left[v_N^2(x)+v_N^2(y)\right]+\frac{2}{\sn_K(d)}\sum_{i =1 }^{N-1}\left(v_i(y)-v_i(x)\right)^2.
 \end{align*}
\end{prop}
\begin{proof} To compute these derivatives, we first compute $\nabla^2_{E_N,E_N}Z(x,y)$. Since $\nabla_{e_N} \gamma'_{x,y}(s) =0$ we immediately have
%Then straightforward computation shows 
\begin{align*}
    \nabla^2 _{E_N,E_N}Z(x,y)=\langle \nabla _{e_N}\nabla _{e_N}\nabla v(y),e_N\rangle -\langle \nabla _{e_N}\nabla _{e_N}\nabla v(x),e_N\rangle.
\end{align*}
 To compute the derivatives $\nabla ^2_{E_i,E_i}Z(x,y)$
 with $i=1,\dots, N-1$, we introduce variations of $\gamma_{x,y}(s)$, which are denoted
 \begin{equation} \label{Definition_of_eta}
     \eta_i(r,s): (-\delta, \delta) \times [-\tfrac{d}{2}, \tfrac{d}{2}] \rightarrow \Omega.
 \end{equation} 
To define this variation, we let $\sigma_1(r)$ be the geodesic with $\sigma_1(0) =x, \tfrac{\partial}{\partial r} \sigma_1(0)=e_i(\tfrac{-d}{2})$ and  $\sigma_2(r)$ be the geodesic with $\sigma_2(0) =y, \tfrac{\partial}{\partial r} \sigma_2(0)=e_i(\tfrac{d}{2})$.
We define $\eta_i(r, s)$, $s \in [-\tfrac{d}{2},\tfrac{d}{2}]$ to be the minimal geodesic connecting $\sigma_1(r)$ and $\sigma_2(r)$.  Since $\Omega$ is strongly convex, the variation $\eta(r,s)$ is smooth.

For fixed $r \neq 0$, the curves $\eta_i(r,\cdot)$ will not be unit speed geodesics in general. Hence we define
   \begin{align*}
    T_i(r,s):=\frac{\eta_i'}{\|\eta_i'\|},
\end{align*}
where we denoted $\partial/\partial s$ by $'$, which is a convention we will use throughout the rest of the paper.

Doing so, we have the following identity:  
\begin{align}  \label{Z_ii1}
   & \nabla^2_{E_i,E_i}Z(x,y)  =\frac{d^2}{dr^2}Z(\eta_i(r,\tfrac{-d}{2}),\eta_i(r,\tfrac{d}{2}))|_{r=0} \\
    &=\langle \nabla_{e_i}\nabla _{e_i}\nabla v(y),e_N\rangle -\langle \nabla_{e_i}\nabla _{e_i}\nabla v(x),e_N\rangle \nonumber\\
   & \quad +2\langle \nabla_{e_i}\nabla v (y),\nabla _r T_i(r,\tfrac{d}{2})\rangle -2\langle \nabla_{e_i}\nabla v(x_0),\nabla _r T_i(r,\tfrac{-d}{2})\rangle  \nonumber  \\
    & \quad+\langle \nabla v(y),\nabla _r\nabla _rT_i(0,\tfrac{d}{2})\rangle - \langle \nabla v(x),\nabla _r\nabla _r T_i(0,\tfrac{-d}{2})\rangle. \nonumber
    \end{align}
We then denote the variation field 
\begin{equation} \label{Definition of J_1}
  J_i(r,s) =  \tfrac{\partial}{\partial r} \eta_i(r,s),
  \end{equation} which is the Jacobi field along $\eta(s)$ satisfying $J_i(r, -\tfrac{d}{2})  =\sigma_1'(r),  \ J_i(r,\tfrac{d}{2})  = \sigma_2'(r)$. For simplicity, we will often drop the initial $0$ and denote $J_i(0,s)= J_i(s)$. 

We can simplify this expression using several formulas derived on \cite[Page 363]{10.4310/jdg/1559786428}. In particular, they showed that for $\gamma = \gamma_{x,y}$
\begin{eqnarray} 
    \nabla _r T_i(0,s)&=&-\langle \gamma',J_i'\rangle e_N+J_i', \label{derivative of V} \\
    \nabla _r\nabla _r T_i(0,s)&=&\Bigl(3\langle \gamma',J_i'\rangle^2 -\|J_i'\|^2-\langle \nabla _r\nabla _r \tfrac{\partial \eta_i}{\partial s},e_N\rangle \Bigr)e_N -2\langle \gamma',J_i'\rangle J_i'+\nabla _r\nabla _r \tfrac{\partial \eta_i}{\partial s}.  \nonumber
\end{eqnarray}
As the tangential component of a Jacobi field is linear and $\langle \gamma',J_i\rangle=0,$ at the end points, we have $\langle \gamma',J_i\rangle=0.$ Therefore $\langle \gamma',J_i'\rangle=0.$
Moreover, since the metric has constant sectional curvature $K,$ we have that the Jacobi fields are given explicitly by 
\begin{align*}
    J_i(s)  = \frac{\cs_K(s)}{\cs_K(\tfrac d2)}e_i(s), 
\end{align*}
so that we find 
\begin{equation} \label{T-derivative}
   \nabla _r T_i(0,s)= -K\frac{\sn_K(s)}{\cs_K(\tfrac d2)}, \ \  \  \nabla _r\nabla _r T_i(0,s)=-K^2\frac{\sn_K^2(s)}{\cs_K^2(\tfrac d2)}e_N+\sum_{j = 1}^{N-1}\langle \nabla _r\nabla _r \tfrac{\partial \eta_i}{\partial s}, e_j\rangle .
\end{equation}
In \cite{10.4310/jdg/1559786428}, it was shown that $\langle \nabla _r\nabla _r \tfrac{\partial \eta_i}{\partial s}, e_j\rangle = 0$ for all $j = 1, \dots, N-1,$ %where the special structure of $\mathbb M^N_K$ was used.
%Here, however, we only assume that the connection we use is the connection associated to the metric $g_{\mathbb M^N_K}.$ We thus need to show that this yet vanishes, which will be postponed to later chapter. 
%\begin{lem}
%    In the notation above, we have that 
%    \begin{align}\label{vanishing-of-variations}
%        \langle \nabla _r\nabla _r \tfrac{\partial \eta_i}{\partial s}, e_j\rangle = 0 \quad \textup{for all} \quad j = 1, \dots, N-1.
%    \end{align}
%\end{lem}
\eqref{T-derivative} and \eqref{Z_ii1}, we have 
% We thus can compute the derivative $\nabla^2_{E_i,E_i}Z(x_0,y_0)$ and obtain 
% where we wrote $J(0,s)=J(s)e_1(s)$ for some function $J=J^{1,1},$ that is $J$ satisfies the ODE $J''+\kappa J=0$ with boundary conditions $J(\tfrac{-d_0}{2})=J(\tfrac{d_0}{2})=1.$
% We thus obtain that
\begin{align*}
    %0 &\geq 
    & \nabla^2_{E_i,E_i}Z(x,y) =\\
    % &=\frac{d^2}{dr^2}|_{r=0}Z(\eta(r,\tfrac{-d_0}{2}),\eta(r,\tfrac{d_0}{2}))\\
    % &=\langle \nabla_{e_1}\nabla _{e_1}\nabla w(y_0),e_2\rangle -\langle \nabla_{e_1}\nabla _{e_1}\nabla w(x_0),e_2\rangle\\
    % &\quad +2\langle \nabla_{e_1}\nabla w (y_0),\nabla _r T(r,\tfrac{d_0}{2})\rangle -2\langle \nabla_{e_1}\nabla w (x_0),\nabla _r T(r,\tfrac{-d_0}{2})\rangle   \\
    % &\quad+\langle \nabla w(y_0),\nabla _r\nabla _rT(0,\tfrac{d_0}{2})\rangle - \langle \nabla w(x_0),\nabla _r\nabla _r T(0,\tfrac{-d_0}{2})\rangle+\nabla^2_{E_1,E_1}F(x_0,y_0)\\
    &\langle \nabla_{e_i}\nabla _{e_i}\nabla v(y),e_N\rangle -\langle \nabla_{e_i}\nabla _{e_i}\nabla v(x),e_N\rangle -2\tn_K(\tfrac d2) \left[v_{ii}(x)+v_{ii}(y)\right]    
   -\tn_K^2(\tfrac d2) Z(x,y).
\end{align*}
We now sum these up for all $ i = 1, \dots, N$ and obtain that 
\begin{align*}
 \sum_{i=1}^N \nabla^2_{E_i,E_i}Z(x,y) 
    &=\langle \Delta \nabla v(y),e_N\rangle -\langle \Delta \nabla v(x),e_N\rangle -2\tn_K(\tfrac d2) \left[\Delta v(y) +\Delta v(x)\right]\\
    &\quad -2\tn_K(\tfrac d2)\nabla_{E_N}Z(x,y)-(N-1)\tn_K^2(\tfrac d2) Z(x,y).
\end{align*}
Then using the identities \begin{eqnarray} 
\nonumber -\lambda \rho +V-\Delta v &= & \|\nabla v\|^2 \label{lap-log} \\
 \Delta\nabla v -\Ric(\nabla v, \cdot) &= & \nabla\Delta v, \label{Bochner}\nonumber
 \end{eqnarray}
 we find that 
 \begin{align*}
\sum_{i=1}^N \nabla^2_{E_i,E_i}Z(x,y) 
    &= -2\tn_K(\tfrac d2)\mathcal C_{[V-\lambda_1\rho]}(x,y)+\mathcal F_{\nabla[ V-\lambda_1\rho]}(x,y)    \\
&  \quad  +2\tn_K(\tfrac d2) \mathcal C_{\|\nabla v\|^2}(x,y)-\mathcal F_{\nabla \|\nabla v\|^2}(x,y)\\
    &\quad -2\tn_K(\tfrac d2)\nabla_{E_N}Z(x,y)-(N-1)\tn_K^2(\tfrac d2) Z(x,y).
 \end{align*}
We now compute $-\mathcal F_{\nabla \|\nabla v\|^2}(x,y)$ to show that   \begin{align*}
    -\mathcal F_{\nabla \|\nabla v\|^2}(x,y) = -2\nabla_{\nabla v(x) \oplus \nabla v(y)}Z(x,y) -2\tn_K(\tfrac d2)\sum_{i = 1}^{N-1}\left[v_i^2(x)+v_i^2(y)\right]+\frac{2}{\sn_K(d)}\sum_{i =1 }^{N-1}\left(v_i(y)-v_i(x)\right)^2.
\end{align*}
We choose the variation $\eta(r,s)$ for the vector $\nabla v(x) \oplus \nabla v(y)$, and the unit variation field  $T$ as before. 
We then have that 
\begin{align*}
    \nabla_{\nabla v(x) \oplus \nabla v(y)}Z(x,y) &= \frac{d}{dr}Z(\eta_i(r,\tfrac{-d}{2}),\eta_i(r,\tfrac{d}{2}))|_{r=0} \\
    &=\langle \nabla _{\nabla v(y)}\nabla v(y),e_N\rangle -\langle \nabla _{\nabla v(x)}\nabla v(x),e_N\rangle \nonumber\\
   & \quad +2\langle \nabla v (y),\nabla _r T(r,\tfrac{d}{2})\rangle -2\langle \nabla v(x_0),\nabla _r T(r,\tfrac{-d}{2})\rangle|_{r=0}.  
\end{align*}
We then let $J(s)$ be the Jacobi field along $\gamma$ with $J(-\tfrac d2) = \nabla v (x), \ J(\tfrac d2) =\nabla v(y)$, that is, 
  $$J(s) = \left(  [\tfrac12-\tfrac{s}{d}]v_N(x)+[\tfrac12+\tfrac{s}{d}]v_N(y)\right) e_N   + \sum_{i=1}^{N-1} \left[v_i(x) \frac{\sn_K(\tfrac d2 -s)}{\sn_K(d)}+v_i(y) \frac{\sn_K(\tfrac d2 +s)}{\sn_K(d)} \right]e_i. $$
  Using \eqref{derivative of V}, we find that
  $$\nabla_{\tfrac{\partial}{\partial r}} T(r, s)|_{r=0} = \sum_{i=1}^{N-1} \left[-v_i(x) \frac{\cs_K(\tfrac d2 -s)}{\sn_K(d)}+v_i(y)\frac{\cs_K(\tfrac d2 +s)}{\sn_K(d)} \right]e_i.$$
Plugging these in, together with the fact that $\nabla_{\nabla v} \nabla v=\frac{1}{2}\nabla\Vert \nabla v \Vert^2$, we get
\begin{align}
2\nabla _{\nabla v(x) \oplus \nabla v(y) }Z(x,y) 
&=\langle \nabla\| \nabla v(y)\|^2 , \gamma'(\tfrac{d}{2}) \rangle-\langle \nabla \|\nabla v(x)\|^2 , \gamma'(\tfrac{-d}{2}) \rangle\nonumber\\
&\quad +2\sum_{i=1}^{N-1}v_i(y)\left(\frac{\cs_K(d)}{\sn_K(d)} v_i(y) -\frac{1}{\sn_K(d)}v_i(x)  \right)\nonumber\\
&\quad -2\sum_{i=1}^{N-1}v_i(x)\left(\frac{1}{\sn_K(d)}v_i(y) -\frac{\cs_K(d)}{\sn_K(d)}v_i(x)  \right).\nonumber
\end{align}
The claim then follows since by straight-forward computation, using $\tn_K(\tfrac d2) = \tfrac{1-\cs_K(d)}{\sn_K(d)}.$
\end{proof}

\section{Proof of Theorem \ref{main-thm} via Stochastic Analysis}
\label{Stochastic Analysis on Manifolds}

In order to prove Theorem \ref{main-thm}, we use the technique of coupled diffusions. This method was previously used in \cite{Gong-Li-Luo2016,cho2023probabilistic} to prove the fundamental gap theorem in Euclidean and spherical geometry. The proof is divided into two steps. We first prove Proposition \ref{thm: super-log-concavity-general} and then Proposition \ref{Gap-Comparison-Improved}, a spectral gap comparison (see also \cite{daifundamental}). The proof of Theorem \ref{main-thm} then follows from combining these two propositions. 
\subsection{Stochastic Analysis}
In order to apply coupled diffusions in this setting, we set some notation. We let $O(M)$ denote the orthonormal frame bundle over $M$ and $\pi : O(M) \rightarrow M$ be the canonical projection. Similar to \cite{hsu2002stochastic}, we define $H_i(U)$ to be the horizontal vector field on $O(M)$, which satisfies  $\pi_*(H_i) = U e_i,$ where $\{e_i\}_{i=1}^N$ is the standard basis on $\mathbb R^N$ and $U \in O(M).$ 

  Following the strategy used in earlier work, we now construct a diffusion $(X_t, Y_t)$ on the product manifold $M\times M$ using the Eells-Elworthy-Malliavin approach. 
  We first construct a diffusion $(U_t,V_t)$ on $O(M)\times O(M)$ and then project it to $M\times M$ via $\pi.$ To introduce the coupling, let $m_{x,y}:T_xM \rightarrow T_yM$ be the \emph{mirror map}, defined as follows: for each $ w \in T_xM,$ $m_{x,y}(w)$ is obtained by first parallel transporting $w$ among $\gamma_{x,y}$ onto $T_yM$ and then reflecting the resulting vector with respect to the hyperplane perpendicular to the geodesic at $y.$ Recall that for $x,y\in \Omega$ we let $\{e_i\}_{1\leq i\leq N}$ be an orthonormal basis at $x$ with $e_N=\gamma'_{x,y}(\tfrac{-d}{2})$ and parallel transport $\{e_i\}_{1\leq i\leq N}$ along $\gamma_{x,y}$. We then have that $m_{x,y}(e_i) = e_i$ for $ i = 1, \dots , N-1$ and $m_{x,y}(e_N) = -e_N.$
  For any function $G: M\times M \rightarrow \mathbb R,$ we denote $\tilde G :O(M)\times O(M)\rightarrow \mathbb R$ to be its lift onto $O(M)$ defined by $\tilde G = G \circ (\pi, \pi).$
  
  In particular, we take $(B_t)_{t\geq 0}$ to be a standard $\mathbb R^N$-valued Brownian motion and for $\alpha \in [0,1]$, we construct the $\alpha$-parametrized family of diffusions on $O(M)\times O(M)$ given by
 \begin{align}\label{SDE-3}
 \begin{cases}
&dU_t=\sum_{i=1}^{N}\sqrt{2}H_{e_i}(U_t)\circ dB^i_t+2H_{\nabla v}(U_t)dt+2\alpha\tn_K\left(\tilde \xi(U_t,V_t)\right)\tilde{\gamma}'(U_t)dt, \quad U_0 = u_x
   \\
&dV_t=\sum_{i=1}^{N}\sqrt{2}H_{e_i}(V_t)\circ dW^i_t+2H_{\nabla v}(V_t)dt-2\alpha\tn_K\left(\tilde \xi(U_t,V_t)\right)\tilde{\gamma}'(V_t)dt,
\quad  V_0 = m_{x,y}u_x  \\
&dW_t=V_tm_{X_tY_t}(U_t)^{-1} dB_t,\\
 &X_t=\pi(U_t),\ \ \ \ \ \ Y_t=\pi(V_t).
 \end{cases}
\end{align} 
Here, we denote $\xi_t=d(X_t,Y_t)/2$ and $\tilde \gamma$ to be the horizontal lift of $\gamma $ onto $O(M)$. Furthermore, we let $u_x$ be the frame $(e_i)_i$ at $x.$

The key ingredient in the stochastic calculus is It\^o's formula  (see, e.g., Chapter 4 in \cite{oksendal2013stochastic}). However, before doing the main computation, we first recall several lemmas from \cite{cho2023probabilistic}. The first shows that if $X_0,Y_0\in \Omega$, the processes $X_t,Y_t$ will stay in $\Omega$ for all $t\geq 0$.
		\begin{lem}\label{lem:process}
		Let $\Omega$ be a bounded strictly convex domain in a Riemannian manifold with smooth boundary $\partial \Omega$.	If $X_0, Y_0\in\Omega$, then $X_t,Y_t\in\Omega$ for all $t\geq 0$.
		\end{lem} 
  \begin{proof}
      The proof is nearly identical to Lemma 3.2 of \cite{cho2023probabilistic}. In order to follow their proof, one needs to justify that $\inf_{\partial \Omega}|\nabla u_1|>0$ which follows from Lemma 3.4 of \cite{gilbarg1977elliptic}.
  \end{proof}
We then consider the 'coupling time` $\tau$, which is defined to be the smallest time at which the diffusions collide
\begin{equation}
    \tau :=
\inf \{t \geq 0 : X_t = Y_t \}.
\end{equation}
The second lemma shows that this time is finite almost surely.
  \begin{lem}\label{lem: finite-coupling-time}
      For any $X_0,Y_0 \in \Omega,$ we have that $\tau <\infty $ almost surely.
  \end{lem}
The proof is identical to the proof of Lemma 3.6 in \cite{cho2023probabilistic}.

 With these two lemmas in hand, we now use It\^o's formula to obtain the following proposition.
\begin{prop}\label{prop: Itos formula for distance}
   For $t < \tau,$ we have that 
      \begin{align}\label{SDE for xi}
          d \xi_t = \sqrt 2d \beta_t+ \left\{-(N-1+2\alpha )\tn_K(\xi_t)+F_t\right\}dt,
      \end{align}
      where $\beta_t$ is a one-dimensional Brownian motion and  $\xi_t = d(X_t,Y_t)/2.$ 
      In particular, for any smooth function $\phi,$ we have that 
      \begin{align}\label{SDE of phi(xi)}
d\phi(\xi_t)=\sqrt{2}\phi'(\xi_t)d\beta_t+\left\{-(N-1+2\alpha)\tn_K(\xi_t)\phi'(\xi_t)+\phi'(\xi_t)F_t+\phi''(\xi_t) \right\}dt.
      \end{align}
  \end{prop}
  \begin{proof}
      \eqref{SDE of phi(xi)} follows from \eqref{SDE for xi} by the standard It\^o's formula. On the other hand, \eqref{SDE for xi} follows from applying It\^o's formula to the lift $\tilde \xi_t = \xi_t\circ (\pi, \pi)$ (see also Section 6.6 in \cite{hsu2002stochastic} for a similar computation) and the well-known second variation of distance in $\mathbb M^N_K$
      \begin{align*}
          \sum_{i = 1}^N{\nabla ^2_{E_i,E_i}}\xi(x,y) = -(N-1) \tn_K(\xi(x,y)).
      \end{align*}
  \end{proof}
%Doing so, we then consider Bochner's horizontal Laplacian, which is defined by 
%\begin{align*}
%    \Delta_{O(M)} = \sum_{i = 1}^N H_i^2.
%\end{align*}

%\begin{prop}[Proposition 3.1.2, \cite{hsu2002stochastic}]
%    For any $f \in C^\infty(M)$ and its lift $ \tilde f = f \circ \pi,$ we have that for any $U \in O(M)$ 
%    \begin{align*}
%        \Delta_{O(M)} \tilde f(U) = \Delta_g f(\pi (U))
%    \end{align*}
%\end{prop}
%This leads us to write a horizontal Brownian motion $\alpha_t$ defined by 
%\begin{align*}
%    \alpha_t = \alpha_0 +\int_0^t H(\alpha_s) \circ d B_s, \quad \alpha_0 \in O(M)
%\end{align*}
%where $B_s$ denotes a standard $\mathbb R^N$ valued Brownian motion.

\subsection{Proof of Proposition \ref{thm: super-log-concavity-general}}
Now that we have defined the relevant stochastic equations, our next goal is to show that the first eigenfunction satisfies a log-concavity estimate. In this section, we use the diffusion $(X_t, Y_t)$ for $\alpha =1.$
  Let $\gamma_{X_t,Y_t}$ be the normal minimal geodesic that goes from $X_t$ to $Y_t$ with $\gamma_{X_t,Y_t}(-\xi_t) = X_t$ and $\gamma_{X_t,Y_t}(\xi_t) = Y_t$.
We define 
\begin{equation}
\label{F_s}
F_t= Z(X_t,Y_t).
\end{equation}
We first derive the SDE for $F_t$, which we obtain from It\^o's formula.% (c.f. Proposition 4.1 \cite{cho2023probabilistic}).
	\begin{prop}\label{prop: SDE for F}
	Let $\lambda = \lambda_1(\Omega, \rho, V)$ be the first eigenvalue of the problem \eqref{eqn: Weighted eigenfunction equation} on a convex domain $\Omega \subset \mathbb M^n_K$ and $v=\log u_1$. Then $F_t$ satisfies
	\begin{align}
	dF_t&=d\{\textup{martingale}\}\nonumber\\
	&\quad +\mathcal F_{\nabla [-\lambda \rho + V]}(X_t, Y_t)+2\textup{tn}_K(\xi_t)\mathcal C_{[\lambda \rho -V]}(X_t, Y_t) dt \nonumber\\
	&\quad  +(N-1)(K-\textup{tn}^2_K(\xi_t))F_tdt\nonumber\\
	& \quad +2\textup{tn}_K(\xi_t)\left(\langle  \nabla v(Y_t), e_N  \rangle^2+\langle \nabla v(X_t), e_N \rangle^2\right)dt\nonumber\\
	&\quad +\frac{2}{\textup{sn}_K(2\xi_t)}\sum_{i=1}^{N-1}\left(\langle  \nabla v(Y_t), e_i  \rangle-\langle \nabla v(X_t), e_i  \rangle \right)^2dt.\label{dF for log-concavity}
	\end{align}
	\end{prop}
Here, `martingale' is a martingale which does not need to be specified for the purposes of this proof. However, it can be written out explicitly.

To show this, we let $\tilde Z $ be the lift of $Z$ onto $O(M) \times O(M).$ In other words, $Z$ satisfies $\tilde Z\circ \pi = Z.$  Since $(U_t,V_t)$ solves the SDE  \eqref{SDE-3}, and since $dF_t = d\tilde Z (U_t,V_t),$ we find that  
     \begin{align*}
         d  F_t &= d \{\textup{martingale}\}+ \left\{\sum_{ i= 1}^N \nabla^2_{E_i,E_i}Z(X_t,Y_t)+2\nabla_{\nabla v(x) \oplus \nabla v(x)}Z+\nabla_{E_N}Z\right\}dt, 
     \end{align*}
     (see \cite{hsu2002stochastic} Section 6 or \cite{cho2023probabilistic} Proposition 4.1 for a similar computation). Then the claim follows from Proposition \ref{Second-order-derivative}. We are now in the position to prove Proposition \ref{thm: super-log-concavity-general}.

\begin{proof}[Proof of Proposition \ref{thm: super-log-concavity-general}]
We first note that from Equation \ref{dF for log-concavity}, since $\overline \rho$ is a modulus of concavity for $\rho$ and $\overline V$ a modulus of convexity for $V,$ we have that
\begin{align}
    dF_t- d\{\textup{martingale}\}\geq &\left\{-2\lambda_1\overline \rho'(\xi_t)+4\lambda_1 \textup{tn}_K(\xi_t)\min_\Omega \rho \right\}dt \nonumber\\
     &+\left\{2\overline V'(\xi_t)-4 \textup{tn}_K(\xi_t)\max_\Omega V \right\}dt \nonumber\\
	& +(N-1)(K-\textup{tn}^2_K(\xi_t))F_tdt\nonumber\\
	& +2\textup{tn}_K(\xi_t)\left(\langle  \nabla v(Y'_t), e_n  \rangle^2+\langle \nabla v(X'_t), e_n  \rangle^2\right)dt\nonumber\\
	\geq&\left\{-2\lambda_1\overline \rho'(\xi_t)+4\lambda_1 \textup{tn}_K(\xi_t)\min_\Omega \rho \right\}dt \nonumber\\
     &+\left\{2\overline V'(\xi_t)-4 \textup{tn}_K(\xi_t)\max_\Omega V \right\}dt \nonumber\\
	& +\left\{(N-1)[K-\textup{tn}^2_K(\xi_t)]+\textup{tn}_K(\xi_t)F_t\right\}F_tdt,
\label{Step1}
\end{align}
where we used the inequality
 \[
    \langle  \nabla v(Y_t), e_N  \rangle^2+\langle \nabla v(X_t), e_N  \rangle^2 \ge \frac{\left(\langle  \nabla v(Y_t), e_N  \rangle - \langle \nabla v(X_t), e_N  \rangle \right)^2 }{2} = \frac{F_t^2}{2}
                \] 
in the second inequality above. 
On the other hand, in view of Proposition \ref{prop: Itos formula for distance}, we get that 
\begin{align}\label{dradial}
&d\left\{(2\psi(\xi_t)+(N-1)\textup{tn}_K(\xi_t)\right\}-d\{\text{martingale}\}\\ 
=&\left\{-2(N+1)\textup{tn}_K(\xi_t)\psi'+2\psi''-(N-1)^2[K+\textup{tn}_K^2(\xi_t)]\textup{tn}_K(\xi_t)+ \left((N-1)(K+\textup{tn}_K^2(\xi_t))+2\psi'\right)F_t \right\}dt.\nonumber
 \end{align}				
Combining \eqref{dF for log-concavity} and \eqref{dradial} using \eqref{condition-2}, we have that
\begin{align*}
&d\left\{F_t-2\psi(\xi_t)-(N-1)\textup{tn}_K(\xi_t)\right\}-d\{\textup{martingale}\}\\
	\geq&\left\{-2\lambda_1\overline \rho'(\xi_t)+4\lambda_1 \textup{tn}_K(\xi_t)\min_\Omega \rho +2\overline V'(\xi_t)-4 \textup{tn}_K(\xi_t)\max_\Omega V\right\}dt \nonumber\\
     &+\left\{ 2(N+1)\textup{tn}_K(\xi_t)\psi'-2\psi''+(N-1)^2[K+\textup{tn}_K^2(\xi_t)]\textup{tn}_K(\xi_t)\right\}dt \nonumber\\
	& +\left\{-2(N-1)\textup{tn}^2_K(\xi_t)-2\psi'(\xi_t)+\textup{tn}_K(\xi_t)F_t\right\}F_tdt\\
 \geq &\left\{ 2(N-1)\textup{tn}_K(\xi_t)\psi'+4\psi\psi'-4\textup{tn}_K(\xi_t)\psi^2+(N-1)^2\textup{tn}_K^3(\xi_t)\right\}dt \nonumber\\
	& +\left\{-2(N-1)\textup{tn}^2_K(\xi_t)-2\psi'(\xi_t)+\textup{tn}_K(\xi_t)F_t\right\}F_tdt\\
 \geq &\left\{ 2(N-1)\textup{tn}_K(\xi_t)\psi'+4\psi\psi'-4\textup{tn}_K(\xi_t)\psi^2+(N-1)^2\textup{tn}_K^3(\xi_t)\right\}dt \nonumber\\
  &+\left\{-2(N-1)\textup{tn}^2_K(\xi_t)-2\psi'(\xi_t)+\textup{tn}_K(\xi_t)F_t\right\}\left(2\psi +(N-1)\textup{tn}_K(\xi_t)\right)dt\\
	& +\left\{-2(N-1)\textup{tn}^2_K(\xi_t)-2\psi'(\xi_t)+\textup{tn}_K(\xi_t)F_t\right\}\left(F_t-2\psi -(N-1)\textup{tn}_K(\xi_t)\right)dt\\
% =&+\left\{-4\textup{tn}_K(\xi_t)\psi^2+(n-1)^2\textup{tn}_K^3(\xi_t)\right\}dt \nonumber\\
  %&+\left\{-2(n-1)\textup{tn}^2_K(\xi_t)+\textup{tn}_K(\xi_t)F_t\right\}\left(2\psi +(n-1)\textup{tn}_K(\xi_t)\right)dt\\
	%& +\left\{-2(n-1)\textup{tn}^2_K(\xi_t)-2\psi'(\xi_t)+\textup{tn}_K(\xi_t)F_t\right\}\left(F_t-2\psi -(n-1)\textup{tn}_K(\xi_t)\right)dt\\
 %=&+\left\{+(n-1)^2\textup{tn}_K^3(\xi_t)-2(n-1)\textup{tn}_K^2(\xi_t)\psi(\xi_t)\right\}dt \nonumber\\
 % &+\left\{-2(n-1)\textup{tn}^2_K(\xi_t)+\textup{tn}_K(\xi_t)F_t\right\}\left((n-1)\textup{tn}_K(\xi_t)\right)dt\\
	%& +\left\{-2(n-1)\textup{tn}^2_K(\xi_t)-2\psi'(\xi_t)+\textup{tn}_K(\xi_t)(F_t+2\psi)\right\}\left(F_t-2\psi -(n-1)\textup{tn}_K(\xi_t)\right)dt\\
 %=&+\left\{-2(n-1)\textup{tn}_K^2(\xi_t)\psi(\xi_t)\right\}dt \nonumber\\
  %&+\left\{-(n-1)\textup{tn}^2_K(\xi_t)+\textup{tn}_K(\xi_t)F_t\right\}\left((n-1)\textup{tn}_K(\xi_t)\right)dt\\
	%& +\left\{-2(n-1)\textup{tn}^2_K(\xi_t)-2\psi'(\xi_t)+\textup{tn}_K(\xi_t)(F_t+2\psi)\right\}\left(F_t-2\psi -(n-1)\textup{tn}_K(\xi_t)\right)dt\\
 =& \left\{-(N-1)\textup{tn}^2_K(\xi_t)-2\psi'(\xi_t)+\textup{tn}_K(\xi_t)(F_t+2\psi)\right\}\left(F_t-2\psi -(N-1)\textup{tn}_K(\xi_t)\right)dt\\
\end{align*}
For convenience, we write $\Psi = \psi +\tfrac{N-1}{2}\textup{tn}_K$ 
\begin{align*}
&d\left\{ e^{\int_{0}^{t}-(N-1)\textup{tn}^2_K(\xi_s)-2\psi'(\xi_s)+\textup{tn}_K(\xi_s)(F_s+2\psi)ds}(F_t-2\Psi(\xi_t))  \right\}\\
&\geq e^{\int_{0}^{t}-(N-1)\textup{tn}^2_K(\xi_s)-2\psi'(\xi_s)+\textup{tn}_K(\xi_s)(F_s+2\psi)ds}d\{\text{martingale}\}.
\end{align*}
Letting $T>0$ be some positive cut-off time, we integrate from $0$ to  $\min\{\tau, T \}$ (denoted $\tau\wedge T$), which gives 
\begin{equation*}
\left( e^{\int_{0}^{\tau\wedge T}-(N-1)\textup{tn}^2_K(\xi_s)-2\psi'(\xi_s)+\textup{tn}_K(\xi_s)(F_s+2\psi)ds}(F_{\tau\wedge T}-2\Psi(\xi_{\tau\wedge T})  \right)\geq \left(F_0-2\Psi(\xi_0)\right)+\text{Martingale}_{\tau\wedge T}.
\end{equation*}
Since $\tau \wedge T$ is bounded, the stopped martingale `$\text{Martingale}_{\tau\wedge T}$' in the above inequality is another martingale. Therefore, we can take the expected value to obtain
\begin{align*}
F_0-2\Psi(\xi_0)&\leq \mathbb{E}\left( e^{\int_{0}^{\tau\wedge T}-(N-1)\textup{tn}^2_K(\xi_s)-2\psi'(\xi_s)+\textup{tn}_K(\xi_s)(F_s+2\psi)ds}(F_{\tau\wedge T}-2\Psi(\xi_{\tau'\wedge T})  \right)\\
&\leq \mathbb{E}\left( e^{\int_{0}^{\tau\wedge T}-(N-1)\textup{tn}^2_K(\xi_s)-2\psi'(\xi_s)+\textup{tn}_K(\xi_s)(F_s+2\psi)ds}|F_{\tau\wedge T}-2\Psi(\xi_{\tau\wedge T})|\right).
\end{align*}
By Fatou's Lemma, we have that
\begin{equation*}
F_0-2\Psi(\xi_0)\leq \mathbb{E}\left( \liminf_{T\rightarrow \infty} e^{\int_{0}^{\tau\wedge T}-(N-1)\textup{tn}^2_K(\xi_s)-2\psi'(\xi_s)+\textup{tn}_K(\xi_s)(F_s+2\psi)ds}|F_{\tau\wedge T}-2\Psi(\xi_{\tau\wedge T})|\right).
\end{equation*}
Since $\tau<\infty$ almost surely, we find that $\tau \wedge T\rightarrow \tau$ as $T\rightarrow \infty$. Since  $F_{\tau}=\Psi(\xi_{\tau})=0$ and since $\tau<\infty$ a.s.
\begin{align*}
&\mathbb{E}\left( \liminf_{T\rightarrow \infty} e^{\int_{0}^{\tau\wedge T}-(N-1)\textup{tn}^2_K(\xi_s)-2\psi'(\xi_s)+\textup{tn}_K(\xi_s)(F_s+2\psi)ds}|F_{\tau\wedge T}-2\Psi(\xi_{\tau\wedge T})| \right)\\
&=\mathbb{E}\left(  e^{\int_{0}^{\tau}-(N-1)\textup{tn}^2_K(\xi_s)-2\psi'(\xi_s)+\textup{tn}_K(\xi_s)(F_s+2\psi)ds}|F_{\tau}-2\Psi(\xi_{\tau})| \right)=0.
\end{align*}
We conclude that
\begin{equation*}
F_0-2\psi(\xi_0)-(N-1)\textup{tn}_K(\xi_0)\leq 0,
\end{equation*}
as desired.
\end{proof}
\subsubsection{Sketch of Maximum Principle Proof}
\label{Sketch of Maximum Principle Proof}

From Proposition \ref{Second-order-derivative}, we can understand the behaviour of $Z$ at a maximal point. In particular, it is straightforward to obtain the following theorem. We simply sketch a proof here.
\begin{prop}
    Under the same assumptions as in Theorem \ref{thm: super-log-concavity-general}, assume that $\psi_0+\tfrac{N-1}{2}\tn_K: [0,D/2] \rightarrow$ is a modulus of concavity of $v.$ Then if $\psi$ is a smooth solution to the problem 
    \begin{align*}
        \begin{cases}
            \tfrac{\partial \psi}{\partial t} \geq \psi'' + 2 \psi'\psi +\overline \lambda\overline \rho' - \overline V'-2\tn_K(s) \left( \psi'+\psi^2+\overline \rho -\overline V\right)\\
            \psi(t,0) = 0\\
            \psi(0, \cdot) = \psi_0,
        \end{cases}
    \end{align*}
    then we have that $\mathcal F_{\nabla v}(x,y) \leq 2 \psi(t,\tfrac d2)+(N-1)\tn_K(\tfrac d2)$ for all $x,y \in \Omega$ and $t \geq 0.$
\end{prop}
\begin{proof}
    Consider $Z_\varepsilon (x,y,t) = \mathcal F_{\nabla v}(x,y) - 2\psi(t, \tfrac d2)-\varepsilon e^{Ct}$ and argue by contradiction, assuming there is a $(x_0,y_0,t_0)$ such that the maximum of $Z_\varepsilon$ is equal to zero.  Then at that point all first order derivatives vanish and the proof is exactly analogous to the one in \cite{10.4310/jdg/1559786428}.  
\end{proof}

\subsection{Spectral Gap Comparison}
We now prove a spectral gap comparison. To do so, we use the diffusions $(X_t, Y_t)$ from \eqref{SDE-3} with $\alpha = 0.$
   \begin{prop}\label{Gap-Comparison-Improved}
       Suppose that $\Omega \subset (M, g_{\mathbb M^N_K})$ is convex and suppose that 
       \begin{align*}
           \mathcal F_{\nabla v}(x,y)\leq 2(\log \overline \varphi_1)' (\tfrac{ d(x,y)}{2})+(N-1)\textup{tn}_K(\tfrac{d(x,y)}{2}).
       \end{align*}
       Then we have that the fundamental gap of the problem \eqref{eqn: one-dimensional-model-weighted} satisfies 
       \begin{align}
           \Gamma(\Omega) \geq \frac{\min \overline \rho}{\max_\Omega \rho}\overline \Gamma(\overline \rho, \overline V).
       \end{align}
   \end{prop}
\begin{proof}
Note that
\begin{align*}
    \Gamma(\Omega, \rho)&=\min _{V\subset H^1, \, \textup{dim}V=2}\max _{f \in V} \frac{\int_{\Omega}|\nabla f|^2 u^2\, dx}{\int_\Omega \rho f^2 u^2\, dx}\\
   &\geq\frac{1}{\|\rho\|_\infty} \min _{V\subset H^1, \, \textup{dim}V=2}\max _{f \in V} \frac{\int_{\Omega}|\nabla f|^2 u^2\, dx}{\int_\Omega  f^2 u^2\, dx}  \\
   &= \frac{\overline \mu}{\|\rho\|_\infty}.
\end{align*}
Let $\overline\varphi_1,\overline\varphi_2$ be the first two eigenfunctions of the corresponding one-dimensional model \eqref{eqn: one-dimensional-model-weighted}. Define $\Phi={\overline\varphi}_2/\overline\varphi_1$ then we have by Proposition \ref{prop: Itos formula for distance} 
\begin{align*}
d\Phi(\xi_t)=\sqrt{2}\Phi'(\xi_t)d\beta_t+\left\{-(N-1)\tn_K(\xi_t)\Phi'(\xi_t)+\Phi'(\xi_t)F_t+\Phi''(\xi_t) \right\}dt.
\end{align*}
By our assumption, we have $F_t = \mathcal F_{\nabla v}(X_t,Y_t)\leq 2(\log \overline\varphi_1)'(\xi_t)+(N-1) \textup{tn}_K(\xi_t)$. 
\begin{align}
\label{Phi without index form}
d\Phi(\xi_t)\leq d\{\textup{martingale}\}+\left\{\Phi'(\xi_t)2(\log\varphi_1)'(\xi_t)+\Phi''(\xi_t) \right\}dt.
\end{align}
Moreover, straightforward computation gives 
\begin{equation}
\label{ODE for Phi}
\Phi''(s)+\Phi'(s)2(\log \overline\varphi_1)'(s)=-\overline \Gamma \overline \rho(s)\Phi(s).
\end{equation} 
Combining \eqref{Phi without index form} and \eqref{ODE for Phi} gives
\begin{equation*}
d\Phi(\xi_t)\leq d\{\textup{martingale}\}-\overline \Gamma \overline \rho\Phi(\xi_t)dt,
\end{equation*}
which is equivalent to 
\begin{equation*}
d\left(e^{\int_0^t\overline \Gamma \overline \rho }\Phi(\xi_t) \right)\leq d\{\textup{martingale}\}.
\end{equation*}
Integrating and taking expectation gives
\begin{equation*}
\mathbb{E}\left(e^{\int_0^t\overline \Gamma \overline \rho}\Phi(\xi_t) \right)\leq \Phi(\xi_0),
\end{equation*}
which gives
\begin{equation}
\label{Bound on Phi(xi_t)}
\mathbb{E}\Phi(\xi_t)\leq e^{-\min \overline \rho \overline \Gamma t}\Phi(\xi_0).
\end{equation}
By Lemma \ref{lem:process}, we know that the above inequality holds for all $t\geq 0$. Finally, denote $w(x)$ to be the Neumann eigenfunction to $\overline \mu$ and letting $w(x,t)=e^{-\overline \mu t}w(x)$ be the solution to the associated heat flow. Then, from the Feynman-Kac formula, we have $ w(t,x)=\mathbb{E}\left[w(X_t)\right]$. Since $w$ is Lipschitz on $\bar{\Omega}$, we can find $\Lambda>0$ such that
				\begin{align}
					|w(x,t)-w(y,t)|&=\left|\mathbb{E}\left[w(X_t)-w(Y_t)\right]\right|\leq\mathbb{E}|w(X_t)-w(Y_t)|\nonumber\\
					&\leq \Lambda \mathbb{E}d(X_t,Y_t)=2\Lambda\mathbb{E}\xi_t. \label{Lipschitz estimate}
				\end{align}
As before, $\Phi'(0)>0$; from which we deduce that there exists $c_1>0$, such that $\Phi(s)\geq c_1 s$ on $[0,D/2]$. Hence, we have from \eqref{Bound on Phi(xi_t)}, that
				\[c_1\mathbb{E}\xi_t\leq  \mathbb{E}(\Phi(\xi_t))\leq e^{-\overline \Gamma \min\overline \rho t}\Phi(\xi_0).  \] 
				Putting this together with \eqref{Lipschitz estimate} and the definition of $w$ gives
				\[e^{-\overline \mu t} |w(x)-w(y)|=|w(x,t)-w(y,t)|\leq \frac{2\Lambda}{c_1}e^{-\overline \Gamma \min\overline \rho t}\Phi(\xi_0).    \]
				Since this must hold for all $t>0$, taking the limit as $t$ gets large gives the claim.
			\end{proof}
\section{The One-Dimensional Model}\label{One-dimensional-comparison}

In this section, we study the fundamental gap in the case $ n= 1$ in order to get quantitative bounds on the fundamental gap. In other words, we let $\overline \rho$ to be a smooth and uniformly positive function, and consider the problem \eqref{eqn: one-dimensional-model-weighted}:
\begin{align*}
    -\varphi ''+ \overline V \varphi= \overline\lambda\overline \rho \varphi\quad \textup{in }[-\tfrac D2, \tfrac D2],   \quad \varphi(-\tfrac D2) = \varphi( \tfrac D2) = 0.
\end{align*}
Using the Raleigh quotient formulation of the first eigenvalue, it is easy to show the following lemma. 
\begin{lem}\label{thm: first-evaule-comparison}
If $\min V \geq -k^2\pi^2/D^2,$
   we have for the $k$-th eigenvalue of the problem \eqref{eqn: one-dimensional-model-weighted}
    \begin{align*} \frac{1}{\max \overline \rho} \left(\frac{k^2\pi^2}{D^2}+\min \overline V\right) \leq \overline \lambda_k \leq  \frac{1}{\min \overline \rho}\left( \frac{k^2\pi^2}{D^2}+\max \overline V\right)
\end{align*}
\end{lem}
\begin{proof} 
    This follows from the Raleigh quotient: to see the first inequality, note that
    \begin{align*}
    \frac{\int (u')^2+\overline Vu^2\, ds}{\int \overline\rho u^2\, ds} \geq \frac{1}{\max \overline \rho}\left(\frac{\pi^2}{D^2}+\min \overline V\right).
    \end{align*}
   One now uses the min max characterization of eigenvalues and the result follows. 
\end{proof}
We now let $L >0$ and we consider the problem 
\begin{align}\label{Euclidean-model-L}
    -\varphi_L'' +\overline V\varphi_L-\lambda_L\varphi_L = 0 \quad \textup{on }[-\tfrac{L}{2}, \tfrac{L}{2}]\quad \&\quad  \varphi(\tfrac{-L}{2}) = \varphi(\tfrac{L}{2}) = 0,
\end{align}
and we denote $\varphi_{L,i}$ to be the $i$-th eigenfunction of that problem and $\lambda_L$ to be the $i$-th eigenvalue. Via a standard Ricatti comparison, one immediately gets the following proposition (for vanishing potential $\overline V \equiv 0)$.
\begin{prop}\label{Ricatti-comparison-prop}
Let $\psi = (\log \overline \varphi_1)'$ and $\psi_L = (\log \varphi_{L,1})'$ both with $\overline V \equiv 0$. Then for any $L \geq \tfrac{\sqrt{\max \overline \rho}}{\sqrt{\min \overline \rho}}D,$ we have that 
\begin{align*}
    \psi \leq \psi_L\quad \textup{on }[-D/2, D/2].
\end{align*}
\end{prop}
\begin{proof}Note that $\psi$ and $\psi_L$ satisfy both a Ricatti equation:
       \begin{align*}
    \psi' + \psi^2 + \overline \lambda \overline\rho &= 0\\
    \psi_L' +\psi_L^2+\overline \lambda_L & = 0,
\end{align*}
with initial condition $\psi_L(0) = \psi(0) = 0.$
By our choice of $L,$ we have that 
\begin{align*}
     \overline \lambda \overline \rho \geq \overline \lambda_L = \frac{\pi^2}{L^2}.
\end{align*}
Applying a Ricatti comparison, we conclude
\begin{align*}
    \psi \leq \psi_L.
\end{align*}
\end{proof}
The following result provides a lower bound on the fundamental gap.
\begin{prop}\label{One-dimensional-fundamental-gap-theorem}
For any potential $\overline V \geq 0$, the fundamental gap of the problem \eqref{eqn: one-dimensional-model-weighted} satisfies \begin{align*}
    \overline \Gamma(\overline \rho, \overline V)\geq \frac{3\pi^2}{D^2}\frac{\min \overline \rho}{(\max \overline \rho)^2}-\left(\frac{\max \overline V}{\min \overline \rho}- \frac{\min \overline V}{\max \overline \rho}\right)
\end{align*}
\end{prop}
\begin{proof}

To see this, note that \begin{align*}
  \frac{\min \overline V}{\max \overline \rho}+\lambda_i(\overline \rho, 0) \leq \lambda_i(\overline \rho, \overline V) \leq \frac{\max\overline V}{\min \overline \rho} +\lambda_i(\overline \rho, 0).
\end{align*}
Hence,
\begin{align*}
    \Gamma(\overline \rho, \overline V) \geq \Gamma( \overline \rho, 0)-\left(\frac{\max \overline V}{\min \overline \rho}- \frac{\min \overline V}{\max \overline \rho}\right).
\end{align*}
To estimate the gap $\overline \Gamma(\overline \rho ) : = \overline \Gamma(\overline \rho, 0)$, we make use of the fact that $\overline w = \tfrac{\varphi_2}{\varphi_1}$ satisfies
\begin{align*}
    \overline w ''+2 (\log \varphi_1)'\overline w ' = - \overline \Gamma (\overline \rho)\overline \rho \overline w \quad \textup{in }[-\tfrac D2, \tfrac D2], \quad \overline w'(-\tfrac D2) = \overline w' (\tfrac D2) = 0.
\end{align*}
Hence we estimate the first Neumann eigenvalue of the operator $-\tfrac{d^2}{ds^2}+2(\log \varphi_1)'\tfrac{d}{ds}.$ To do so, observe that 
\begin{align*}
    \overline \Gamma(\overline \rho) &= \min_{V: \,\textup{dim}(V) = 2} \max_{\phi \in V}\frac{\int_{-\tfrac D2}^{\tfrac D2}(\phi')^2\varphi_1^2\, ds}{\int_{-\tfrac D2}^{\tfrac D2}\overline \rho\phi^2\varphi_1^2\, ds}\\
    & \geq \frac{1}{\max \overline \rho} \min_{V: \,\textup{dim}(V) = 2} \max_{\phi \in V}\frac{\int_{-\tfrac D2}^{\tfrac D2}(\phi')^2\varphi_1^2\, ds}{\int_{-\tfrac D2}^{\tfrac D2}\phi^2\varphi_1^2\, ds}\\
    & = \frac{1}{\max \overline \rho}\overline \mu,
\end{align*}
where $\overline \mu$ is the first non trivial Neumann eigenvalue of the problem 
 \begin{align}\label{Neumann-problem-one-dimensional}
        \overline w ''+2 (\log \overline\varphi_1)'\overline w ' = -\mu \overline w \quad \textup{in }[-\tfrac D2, \tfrac D2], \quad \overline w'(-\tfrac D2) = \overline w' (\tfrac D2) = 0.
 \end{align}
 We now set $\psi = (\log\overline \varphi_1) '$ and $\psi _L = (\log \varphi_{L,1})',$ where $\varphi_{L,1}$ is the first eigenfunction of the problem \eqref{Euclidean-model-L}. We then see from Proposition \ref{Ricatti-comparison-prop} that
\begin{align*}
    \psi \leq \psi_L.
\end{align*}
The claim then follows from similar arguments as in Theorem \ref{Gap-Comparison-Improved}. Alternatively, it follows directly from Proposition 3.2 in \cite{andrews2011proof}. 
\end{proof}
\section{Horoconvex Domains and Their Fundamental Gap}\label{Section: horoconvex-estimates}

\begin{proof}[Proof of Theorem \ref{The horoconvexity theorem}]
    
We consider the Poincar\'e model
\begin{align*}
    g_{\mathbb H^N} = \frac{4}{(1-\|x\|^2)^2}g_{\mathbb R^N}.
 \end{align*}
In view of \eqref{Eigenvalue-under-conformal-change-explicit-scalarcurvature}, we have 
\begin{align*}
    \Gamma(\Omega, \Delta_{g_{\mathbb H^N}}) = \Gamma(\Omega, \Delta_{\mathbb R^N}, e^{2\varphi})
\end{align*}
where \begin{align*} \exp(2\varphi) = \frac{4}{(1-\|x\|^2)^2}, \quad \textup{i.e. }\quad 
    \varphi = \log \left(\frac{2}{1-\|x\|^2}\right).
\end{align*}
To apply Theorem \ref{main-thm}, we need to construct a modulus a modulus of concavity for $\rho = \exp (2 \varphi).$
We calculate the eigenvalues of the Hessian of $\rho$ are given (in polar coordinates)
 \begin{align*}
     \sigma^\rho_1 =   \frac{16}{(1-r^2)^3}, \quad \sigma^\rho_2 =  \frac{16(1+5r^2)}{(1-r^2)^4}.
 \end{align*}
    Let $R$ denote the circumradius in hyperbolic geometry. Then, we denote $R_E = \tanh(\tfrac{R}{2})$ its Euclidean radius.
   We have that \begin{align}\label{Defn: of overline rho}
    \overline \rho_C (s) = \sup_\Omega\max\{\sigma^\rho_1, \sigma^\rho_2\}\frac{s^2}{2}+C=(\max_\Omega \rho)^2(1+5R_E^2)\frac{s^2}{2} +C 
\end{align}
is a modulus of concavity for any 
constant $C>0.$ In particular, 
\begin{align}\label{ineq: bounds for overlinerho}
   C \leq  \overline \rho \leq \tfrac{3}{4}(D_E\max _\Omega \rho)^2 +C,
\end{align}
where $D_E$ denotes the diameter of $\Omega$ with respect to the Euclidean metric. 
Since $\rho' >0,$ we need
that the eigenvalue of the problem \eqref{eqn: one-dimensional-model-weighted} with $\overline \rho = \overline \rho_C$ satisfy the condition \eqref{condition-2}.
To ensure this, we fix a number $C>0,$ and choose 
\begin{align}\label{Defn: of overline V}
    \overline V \equiv  \max_{s\in [0,D/2]} \overline \rho(s)\, \left(\lambda_1(\Omega, \Delta_{\mathbb R^N}, e^{2\varphi}) -\overline \lambda(\overline \rho, 0)\right).
\end{align}
To make the choice of $\overline V$ depend only on the hyperbolic diameter of $\Omega$, note that
\begin{align*}
    D_{E }\leq \frac{D_{\mathbb H^n}}{2},
\end{align*}
as one can see by computing length of curves and taking infimums. To remove the dependence on $\lambda_1(\Omega, \Delta_{\mathbb R^N}, \rho),$ we recall that for $\rho= e^{2\varphi}$
\begin{align*}
    \lambda(\Omega,\Delta_{\mathbb R^N}, \rho) = \lambda_1(\Omega, \Delta_{\mathbb H^n})-\frac{N(N-2)}{4}\leq \lambda_1(B_r, \Delta_{\mathbb H^n})-\frac{N(N-2)}{4},
\end{align*}
where $B_r\subset \Omega$ is the largest ball inscribed in $\Omega.$

For horoconvex domains, Theorem 1 in \cite{Borisenko1999} shows that the inradius will be bounded from below by the diameter. More precisely, we have that
\begin{align*}
    \frac{D}{2} \leq \xi(r) +r,
\end{align*}
where \begin{align*}
   \xi(r) =  \ln\left(\frac{(1+\sqrt\tau)^2}{1+\tau}\right)\quad \textup{and } \quad  \tau(r) = \tanh (\tfrac r2).
\end{align*}
To find the lower bound, note that $\xi(r) \leq 2 \sqrt r$ and thus  
\begin{align}\label{ineq: diameter-inradius}
r \geq \left(-1+\sqrt{\tfrac D2+1}\right)^2.
\end{align}
One now employs an upper bound for the first Dirichlet eigenvalue of balls in hyperbolic space (see e.g. Theorem 5.6 in \cite{savo2009lowest}) to find that 
\begin{align*}
    \lambda_1(B_r, \Delta_{\mathbb H^N}) \leq \frac{(N-1)^2}{4}+ \frac{\pi^2}{r^2}+ \frac{C}{r^3},
\end{align*}
where $C= \tfrac{\pi^4(N^2-1)}{12}.$
Putting everything together, we obtain that 
\begin{align}\label{ineq: upper bound for first eigenvalue}
    \lambda_1(\Omega, \Delta_{\mathbb R^N}, \rho) \leq \frac{1}{4}+\frac{\pi^2}{\left(-1+\sqrt{\tfrac D2+1}\right)^4}+ \frac{(N^2-1)\pi^4}{12\left(-1+\sqrt{\tfrac D2+1}\right)^6}.
\end{align}
Hence we can choose $\overline V$ to be greater or equal than a constant, depending only on the diameter of $\Omega.$
Then, all the assumptions of Theorem \ref{main-thm} are satisfied and we conclude that 
\begin{align}\label{two-point-estimate-for-horoconvex-functions}
    \mathcal F_{\nabla v}(x,y) \leq 2(\log \overline\varphi_1)'(\tfrac d2)
\end{align}
By Theorem \ref{Gap-Comparison-Improved}, we conclude that 
\begin{align*}
    \Gamma(\Omega)& \geq\frac{\min \overline \rho}{\max_\Omega \rho}\overline \Gamma(\overline \rho, \overline V).
\end{align*}
\end{proof}

\subsection{Explicit bounds on the gap}\label{section: rough-bounds-for-horoconvex}
In this subsection, we prove closed-form estimates on the spectral gap. The idea is to use the bound of Proposition \ref{thm: super-log-concavity-general} to derive a bound on the Hessian and then employ Theorem 3 of \cite{charalambous2015eigenvalue}. From Proposition \ref{thm: super-log-concavity-general}, we can divide \eqref{two-point-estimate-for-horoconvex-functions} by $d$ and passing to the limit as $d \rightarrow 0^+$. Doing so, we find that
\begin{align*}
    \nabla^2 v \leq \psi'(0),  
\end{align*}
where $\psi = (\log \overline \varphi_1)'.$ Since $\overline \rho$ is even, we have that $\psi$ is odd and hence we get that 
\begin{align*}
 \psi'(0) = -\psi^2(0) -\overline \lambda \overline \rho (0)+ \overline V = -\overline\lambda (\overline \rho, \overline V) \min \overline \rho + \overline V.
\end{align*}
To estimate the right-hand side of the above inequality, note that by Theorem \ref{thm: first-evaule-comparison}
\begin{align*}
    -\overline\lambda (\overline \rho, \overline V) \min \overline \rho + \overline V & \leq -\overline \lambda(\overline \rho,0)\min \overline \rho +\overline V \left(1-\frac{\min \overline \rho}{\max \overline \rho}\right) \\
    & \leq -\overline \lambda(\overline \rho,0)\max \overline \rho + \frac{3}{4}(D_E^2\max _\Omega \rho)^2 \lambda_1(\Omega, \Delta_{\mathbb R^N}, \rho).
\end{align*}
We now estimate $\frac{3}{4}(D_E^2\max _\Omega \rho)^2 \lambda_1(\Omega, \Delta_{\mathbb R^N}, \rho)$ in terms of the diameter of $\Omega.$ Note that $\max _\Omega \rho = 4/(1-R_E^2)^2,$ where $R_E$ is the Euclidean circumradius of $\Omega.$ A result by Dekster \cite{dekster1995jung} shows that
\begin{align*}
    R_{\mathbb H^N} \leq \textup{arcsinh}\left(\frac{\sqrt{2N}}{\sqrt{N+1}}\sinh(D)\right),
\end{align*}
 where we shorten $D_{\mathbb H^N}$ to $D$.

From this, we find that
\begin{align*}
    R_E \leq \tanh\left(\frac{1}{2}\textup{arcsinh}\left(\frac{\sqrt{2N}}{\sqrt{N+1}}\sinh(D)\right)\right)
\end{align*}
and further estimate
\begin{align*}
    \frac{1}{(1-R_E^2)} &\leq \cosh^2\left(\tfrac{1}{2}\textup{arcsinh}\left(\tfrac{\sqrt{2N}}{\sqrt{N+1}}\sinh(D)\right)\right)\\
    & \leq  \cosh^2\left(\textup{arcsinh}\left(\tfrac{\sqrt{2N}}{\sqrt{N+1}}\sinh(D)\right)\right)\\
    & =1+ \tfrac{{2N}}{{N+1}}\sinh^2(D).
\end{align*}
We thus get that 
\begin{align}\label{estimate-of-maximum}
  \tfrac{3}{4}(D_E \max_\Omega \rho) ^2 \leq 3(4 \wedge D^2)\left( 1+ \tfrac{{2N}}{{N+1}}\sinh^2(D)\right)^4,
\end{align}
 where here and in the following we denote $a\wedge b: = \min\{a,b\}$ for any $a,b\in \mathbb R.$ 

Lastly, we use the estimate from Theorem \ref{thm: first-evaule-comparison} and find that $-\overline \lambda(\overline \rho) \max \overline \rho \leq -\frac{\pi^2}{D_E^2}\leq -\pi^2(\frac{4}{D^2} \wedge 1)$ (we again use the notation $\wedge$ to indicate the maximum)
Putting these together, we find that 
\begin{align*}
    \nabla^2 v \leq -{\pi^2} (1 \wedge \frac{4}{D^2})+ 3(4 \wedge D^2)\left( 1+ \tfrac{{2N}}{{N+1}}\sinh^2(D)\right)^4\left( \frac{1}{4}+\frac{\pi^2}{\left(-1+\sqrt{\tfrac D2+1}\right)^4}+ \frac{(N^2-1)\pi^4}{12\left(-1+\sqrt{\tfrac D2+1}\right)^6}\right).
\end{align*}
This upper bound is only in terms of the diameter and the dimension, so we denote the right hand side of the inequality as $R(N, D).$ We finally conclude that 
\begin{align*}
    \Gamma(\Omega, \Delta_{\mathbb H^N}) &=
   \min _{V\subset H^1, \, \textup{dim}V=2}\max _{u \in V} \frac{\int_{\Omega}|\nabla u|^2 e^{-f}\, dx}{\int_\Omega \rho u^2 e^{-f}\, dx}\\
   &\geq\frac{1}{\|\rho\|_\infty} \min _{V\subset H^1, \, \textup{dim}V=2}\max _{u \in V} \frac{\int_{\Omega}|\nabla u|^2 e^{-f}\, dx}{\int_\Omega u^2 e^{-f}\, dx}  = \frac{\mu}{\|\rho\|_\infty}.
\end{align*}
The Neumann eigenvalue $\mu$ can now be estimated using Theorem 3 of \cite{charalambous2015eigenvalue} and we conclude that 
\begin{align}\label{explicit-lower-bound-Step1}
    \Gamma(\Omega) \geq \frac{\pi^2}{(4\wedge \tfrac{D^2}{2})\left(1+ \tfrac{{2N}}{{N+1}}\sinh^2(D)\right)^2}\exp\left(-C_N(1\wedge \tfrac D2)\sqrt{R(N,D)}\right).
\end{align}

\subsubsection{Asymptotic bounds on the gap}

\label{But why is the gap small for small diameter?}

For $D \gg N >2$, it is possible to rewrite this inequality completely explicitly. In particular, using the derivation from Theorem 3 of \cite{charalambous2015eigenvalue}, after some algebraic manipulations we obtain the asymptotic estimate 
\begin{equation} \label{Asymptotic gap bound}
  \Gamma(\Omega) >  \frac{\pi^2(N-1)^2 D^2} {16} \exp\left[ - (N-1) (D^2) \left(1+ 2 \exp(2D) \right)^2  \right].
\end{equation} 

From this, we obtain a bound which decays at doubly-exponential rate in terms of the diameter. Nguyen, Stancu and Wei \cite{nguyen2022fundamental} showed that the gap of large horoconvex domains is bounded from above by $\frac{C_N}{D^3}$, so these two results raise the following question.
\begin{quest}
    How quickly does the fundamental gap of a horoconvex domain decay in terms of its diameter?
\end{quest}

The gap estimate we obtain in \eqref{explicit-lower-bound-Step1} also decays when $D$ is small, since in this case the final term in $R(N,D)$ is dominant and tends to infinity. This might be somewhat unexpected, since the fundamental gap conjecture posited that in $\mathbb R^N,$ $\Gamma(\Omega) \geq \frac{3 \pi^2}{D^2}$. We will provide an estimate for small horoconvex domains in Section \ref{Small horoconvex domains estimates} which is asympotically of this size. However, let us explain the reason for why the gap estimate in the previous section decays as the diameter does.

Although our focus in this section has been horoconvex domains, the modulus of concavity estimate applies to any domain which is convex with respect to the Poincar\'e disk model. In particular, the entire argument (except for Equation \eqref{ineq: diameter-inradius}) applies to thin rectangles in the disk model. In other words, such domains have the Hessian upper bound: 
\begin{align*}
    \nabla v^2\leq -{\pi^2} (1 \wedge \frac{4}{D^2})+ 3(4 \wedge D^2)\left( 1+ \tfrac{{2N}}{{N+1}}\sinh^2(D)\right)^4\lambda_1(\Omega, \Delta_{g_{\textup{E}}, e^{2\varphi}})
\end{align*}
which blows up for thin rectangles. 

For such domains, one can essentially replicate the argument of \cite{khan2022negative} to show that the fundamental gap can go to zero as the inradius tends to zero (see Appendix \ref{We need the inradius} for details). As such, it is not possible to obtain a uniform estimate for  the fundamental gap for domains which are convex in the Poincar\'e disk model solely in terms of the diameter and the dimension. In particular, the estimate must incorporate the inradius (or equivalently, the diameter). As the diameter goes to zero, the inradius must also go to zero. Therefore, the estimate we obtain must tend to zero as the diameter shrinks in order to give a valid theorem. Therefore, any proof which obtains more refined bounds of the gap of horoconvex domains can only apply to a more restrictive class of domains.

\begin{comment}
\note{The following comment just derives the asymptotic bound as the diameter of a horoconvex domain gets large }
The bound from \cite{charalambous2015eigenvalue} is the following.
  \begin{eqnarray}  \mu_1 &\geq& \frac{\pi^2}{16} \cdot \frac{\max (n-1,2)(n-1) k}{\left(\exp \left(1 / 2 \max (\sqrt{n-1}, \sqrt{2}) \sqrt{(n-1) k d^2}\right)-1\right)^2} \\
  & = & \frac{\pi^2}{16} \frac{ (n-1)^2 k}{\left(\exp \left(1 / 2 (n-1) \sqrt{k} d \right)-1\right)^2} \\
  &\geq & \frac{\pi^2}{16} \frac{ (n-1)^2 k}{\exp \left( (n-1) \sqrt{k} d  \right)}
  \end{eqnarray}
  Meanwhile, the estimate on the Hessian of $v$ shows that for $D$ large, 
  \begin{equation}
      k < D^2 (1+2 \exp(2D))^4
  \end{equation}
where I'm using $1>\frac34$ to absorb some of the lower order terms. Furthermore, the computation on $R_E$ shows that $\| \rho\| \leq 4(1+ 2 \exp(2D))^2$.

Therefore, we have that
\[ \frac{\mu}{\rho}  \geq \frac{\pi^2}{16} \frac{(N-1)^2 D^2}{\exp\left((N-1) D^2 (1+ 2 \exp(2D))^2 \right)}\]
Doing one more large $D$ simplification, we have that
  \[ \frac{\mu}{\rho}  \geq \frac{\pi^2}{16} \frac{(N-1)^2 D^2}{\exp\left((4+o(1))(N-1) D^2  \exp(4D) \right)}\]
\end{comment}

\section{Fundamental Gap Estimates on Conformally Flat manifolds}\label{Section: general-conformally-flat}

The approach we presented in Section \ref{Section: horoconvex-estimates} also works for more general conformally flat manifolds. In particular, Proposition \ref{prop: conformally-flat-gap-scalarcurvature} shows that for any domain which is convex with respect to the flat connection, we have that $\Omega \subset (M, g = e^{2\varphi}g_{\textup{E}}),$ that 
\begin{align*}
    \Gamma(\Omega, \Delta_{g}) \geq \Gamma(\Omega, \Delta_{g_{\textup{E}}}, e^{2\varphi})-\frac{N-2}{4(N-1)}\textup{osc}_\Omega(R_g).
\end{align*}
We are now in the position to present the 
\begin{proof}[Proof of Theorem \ref{thm: Gap-Estimate-for-conformally-flat-manifolds}]
    To estimate $\Gamma(\Omega, \Delta_{g_{\textup{E}}}, e^{2\varphi})$ we let $\overline \rho$ be a modulus of concavity for $\rho = e^{2\varphi}$ with respect to $g_{\textup{E}}.$ In other words, we let $\sigma_{\max} = (\max_\Omega \sigma(\nabla^2 \rho))$ be the largest eigenvalue of the Hessian of $\rho$ with respect to the Euclidean metric. If $\sigma_{\max}$ is negative (i.e., $\rho$ is concave), we replace $\sigma_{\max}$ by $0$, which will ensure that $\overline{\rho}$ is constant and that \ref{condition-2} holds automatically. In other words, $\sigma_{\max} := (\max_\Omega \sigma(\nabla^2 \rho) \vee 0)$.
    We then set 
    \begin{align*}
        \overline \rho(s) = \sigma_{\max} \tfrac{s^2}{2}+C
    \end{align*}
    for some constant $C>0.$
    We then choose $\overline V$ as in Section \ref{Section: horoconvex-estimates} and the whole proof goes through so long as 
    \begin{align}\label{ineq: gap-estimate-for-conformally-flat}
        \Gamma(\Omega, \Delta_g) \geq \frac{\min \overline \rho}{\max_\Omega \exp(2\varphi)}\overline \Gamma(\overline \rho, \overline V)-\frac{N-2}{4(N-1)}\textup{osc}_\Omega(R_g).
    \end{align}
\end{proof}

Note that when $\sigma_{\max}$ is positive, the quantity $\overline{V}$ depends on $\lambda_1(\Omega)$, and this dependence on is crucial since if the domain collapses to a line, the gap can become arbitrarily small (see Appendix \ref{We need the inradius} for details).  However, if one restrict to domains $\Omega$ whose inradius $r_\Omega$ is bounded from below by a positive function of the diameter $D,$ one can make the estimate to be dependent only on the diameter of $\Omega$ and the conformal factor.
%To this end, we consider the class of domains $\mathcal G $ 
%\begin{align*}
% \mathcal G_f = \{ \Omega \subset (M, g = e^{2\varphi}g_{\textup{E}})\, | \, \Omega\quad \textup{is convex with respect to }\, g_{\textup{E}} \quad \textup{and }
%r_{\Omega} \geq f(D)\},\end{align*}
%where $f$ is a given positive function. This establishes the following.
\begin{cor}\label{Estimate-for-inradius-bounded-below}
   Suppose that $\Omega$ is convex with respect to $g_{\textup{E}}$ and whose  inradius is bounded below by some positive function of the diameter. Then there exist $\overline \rho$ and $\overline V$ depending only on the conformal factor and $D$ such that $\eqref{ineq: gap-estimate-for-conformally-flat}$ holds.
\end{cor}

When $\rho$ is concave, we obtain a gap estimate which is independent of the inradius (c.f. Theorem 4.1 of \cite{khan2024concavity}). In particular, we obtain the estimate
    \begin{align}\label{Estimate-for-rho-concave}
        \Gamma(\Omega) \geq \frac{\min \exp(2\varphi)}{\max\exp (2\varphi)}\frac{3\pi^2}{D^2}-\frac{(N-2)}{4(N-1)}\textup{osc}_\Omega(R_g).
    \end{align}
    
\subsubsection{A worked example: $\mathbb S^1 \times \mathbb S^{N-1}$}
\label{gaps of S^1 times S^2}

As a demonstration of how this theorem can be applied, we consider $M = \mathbb S^1\times \mathbb S^N$ with its standard metric, which is conformally-flat. When $N$ is even, we can also consider Hopf manifolds. 

For domains $\Omega \subset \mathbb S^1 \times \mathbb S^{N-1}$, it will be convenient to consider the affine universal cover $\mathbb{R}^N \backslash 0$ with the conformal metric $g_{\mathbb S^1 \times \mathbb S^{N-1}} = \frac{1}{r^2} g_{\textup{E}}$, where $r = \|x\|^2.$  
When $N$ is even, we can consider $\mathbb{R}^N \backslash 0$ as $\mathbb{C}^{N/2} \backslash 0$ and construct Hopf manifolds as the quotient of this cover by the action of $\mathbb{Z}$ generated by an holomorphic contraction $z \to \alpha z$ where $\alpha \in \mathbb{C}$ satisfies $|\alpha|<1$. The class of domains $\Omega \subset \mathbb{R}^N \backslash 0$ that we consider are those which are convex and compact with respect to the affine structure.

Then, a straightforward computation shows that the eigenvalues of 
the Hessian of $\frac{1}{r^2}$ are $\frac{-2}{r^4}$, repeated with multiplicity $N-1$ and $\frac{6}{r^4}$ with multiplicity $1$. Furthermore, since $\mathbb{S}^1 \times \mathbb{S}^{N-1}$ has constant scalar curvature, bounding the gap reduces to bounding the gap of \eqref{eqn: only-weighted-no-potential}.
As before, we construct a modulus of concavity $\overline \rho $ given by 
    \begin{align*}
        \overline \rho(s) = \frac{6}{\inf_{x \in \Omega} r^4} \tfrac{s^2}{2}+C,
    \end{align*}
    where $C$ is a positive constant which we can take to be $\inf_{x \in \Omega} \frac{1}{r^2}$ %

    In order to obtain a gap estimate, we must verify that $\overline \lambda(\overline \rho, \overline V) \geq \lambda_1(\Omega,  \rho),$ for $\rho = \tfrac{1}{r^2}.$ To do so, we choose $\overline V$ large enough so that the desired eigenvalue comparison holds. Doing so, we will get an upper-bound on the log-concavity of $u_1$, and thus a lower bound on the fundamental gap of the domain in terms
    of the diameter and the principle eigenvalue of the domain.
    
    \begin{comment}
    From this, we find that Theorem \ref{main-thm} that
    \begin{align*}
    \Gamma(\Omega, \rho) \geq \frac{\min \overline \rho}{\max_\Omega \rho}\Gamma(\overline \rho, \overline V).
    \end{align*}
        \end{comment}

\subsection{Gap estimates for small horoconvex domains}

\label{Small horoconvex domains estimates}

In \cite{khan2024concavity}, we considered domains which are convex with respect to a particular \emph{spherical geometry}, which is a more restrictive class of domains, though general enough to include horoconvex domains up to a certain diameter. We can again use positively curved geometry as the reference, and doing so well yield stronger gap estimates compared to our earlier work when the diameter is small.

To do this, we again consider the Poincar\'e disk model $(\mathbb H^N, g_{\mathbb H^N}) = (B_1(0), \tfrac{4}{(1-r^2)^2}g_{\textup{E}})$ for the hyperbolic space and consider a sphere of radius $R$ (i.e. $K = 1/R^2).$ Using stereographic projection map from the disk model to the sphere, we can relate the hyperbolic metric to a spherical one by the conformal change $g_{\mathbb H^N} = \tfrac{(R^2+r^2)^2}{R^4(1-r^2)^2}g_{\mathbb S^N_K}$ on the ball $B_R(0) \subset B_1(0).$ For diameters $D$ small enough, one can repeat the argument of the proof of Theorem \ref{thm: small-deformation-of-sphere} to obtain the estimate
  \begin{align}\label{Gap-Estimates-for-small-horoconvex-domains}
         \Gamma(\Omega) \geq  \frac{(\min_\Omega \rho-\sigma_{\max}\tfrac{D^2}{8})^3}{\min_\Omega \rho(\max_\Omega \rho)^2}\frac{3\pi^2}{D^2}-\frac{N(N-2)}{4}\textup{osc}_\Omega(\rho^{-1}),
    \end{align}
    where $\rho = \exp(2\varphi)=\tfrac{(R^2+\|x\|^2)^2}{R^4(1-\|x\|^2)^2}$ and $\sigma_{\max}$ chosen as in \eqref{choice-of-sigma-max}. By taking the diameter of $\Omega$ small and choosing the radius $R>0$ small enough, the modulus of concavity $\overline \rho$ of $\rho$ will satisfy the assumptions of Theorem \ref{main-thm}.
In particular, we can construct a modulus of concavity $\overline \rho $ of $\rho$ given by 
    \begin{align*}
        \overline \rho(s) = \sigma_{\max} \tfrac{s^2}{2}+C.
    \end{align*}
  In order to apply the comparison, we can choose $D>0$ small enough and $K>0$ large enough (i.e. $R$ small enough) such that for $C = \min _\Omega \rho-\tfrac{D^2}{8}\sigma_{\max}>0,$ $\overline \rho$ satisfies 
    \begin{align*}
        \overline \rho \leq \min \exp(2\varphi) \quad \& \quad \overline \rho' \leq 2\tn_K \overline \rho.
    \end{align*}
Thus \eqref{Gap-Estimates-for-small-horoconvex-domains} follows.
%The last part of this argument is similar to the one provided in the proof of the main result in \cite{khan2024concavity}.

\subsection{Spectral Gap Estimates on Surfaces of Positive Curvature}

In \cite{andrews2014moduli}, Andrews posed the following question.

\begin{quest}[Page 19 \cite{andrews2014moduli}]
   [C]an one expect a useful lower bound on the gap on
a strongly convex Riemannian manifold by controlling curvature, or is it necessary to control
higher derivatives of curvature as well?
\end{quest}

Since this question was first posed, there have been a number of developments. For instance, the curvature must be non-negative in order to establish gap estimates for Levi-Civita convex domains. However, the situation in higher-dimensional manifolds remains open, even for well-understood spaces like $\mathbb{CP}^n$ (see \cite{aryan2024concavity} for some recent progress on this topic).

Nonetheless, it is possible to obtain gap estimates for surfaces whose curvature is positive and controlled in a $C^2$ sense (see Theorem 1.1 of \cite{surfacepaper1, khan2023modulus}). But this leaves unresolved the question of whether it is possible to obtain gap estimates when the curvature is simply positive. Although it involves a different notion of ``convexity" for domains, the results in this paper give evidence that positive curvature suffices to establish gap estimates.

Indeed, suppose that $(M^2,g)$ is a surface of Gaussian curvature $\kappa$. From the uniformization theorem, we can locally write $g = e^{2\varphi }g_{\textup{E}}$ where $\varphi$ is a function satisfying \eqref{eqn: Scalar-curvature-under-conformal-deformation}, which simplifies to \[\kappa=-2(N-1) e^{-2\varphi} \Delta_{\textup{E}}\varphi.\] If $e^{2\varphi}$ is concave with respect to $g_{\textup{E}},$ i.e. if 
\begin{align*}
    \Hess \varphi +2d\varphi \otimes d\varphi \leq 0,
\end{align*} then we can let $\overline \rho \equiv 1$ which is a modulus of concavity for $\rho = e^{2\varphi}$. In that case, for any $\Omega\subset M$ that is convex with respect to $g_{\textup{E}}$, we find that 
\begin{align}\label{concavely-curved surfaces}
    \Gamma(\Omega) \geq \frac{\min \exp(2\varphi)}{\max \exp(2\varphi)}\frac{3\pi^2}{D^2}.
\end{align}
So we have that whenever the conformal factor is concave, one can prove strong gap estimates. Note that the curvature could be positive even if the conformal factor is not concave (but $\Delta_{\textup E} \varphi <0$). For these situations, our approach leads us to prove lower bounds on the fundamental gap involving the first eigenvalue (as above) but not solely the diameter.

\begin{comment}

Note that this estimate differs from the estimate given in \cite{surfacepaper1,khan2023modulus}. In fact, there the authors along with Wei showed (among other things) that there is $\alpha= \alpha(\underline \kappa,\overline \kappa) >0$ such that for any surface $(M^2,g)$ whose sectional curvature $\kappa>0,$ satisfies $\underline \kappa \leq \kappa \leq \overline \kappa$ and that $|\Delta \kappa|_\infty, |\nabla \kappa|_\infty \leq \alpha(\underline \kappa, \overline \kappa) $ then for any $g$ convex set $\Omega$ of diameter $D \leq \tfrac{\pi}{2\sqrt{ \overline \kappa}}$ we have 
\begin{align*}
    \Gamma(\Omega) \geq \frac{3\pi^2}{D^2}-(12+3\pi)\left(\overline \kappa-\underline \kappa\right).
\end{align*}
The assumption of the latter estimate is a $C^2$ assumption on $\kappa$ and hence a $C^4$ assumption on $g.$ On the contrary, the estimate \eqref{concavely-curved surfaces} we provide here, does only need the conformal factor to be concave. On the other hand, in Theorem \ref{thm: small-deformation-of-sphere}, we also prove such an estimate for small deformations of spheres in dimension $2,$ only assuming a $C^2$ condition on $g.$
    
\end{comment}

\subsection{Fundamental gaps on conformal deformations of round spheres}
Another application of Proposition \ref{thm: super-log-concavity-general} is to obtain fundamental gap estimates of domains in small conformal deformations of round spheres (c.f. Corollary 1.4 of \cite{khan2024concavity}\footnote{The primary difference between the estimates in our earlier work is that the leading term in the lower bound here is roughly three times as large.}). 

In this section we study the case where $M$ is the sphere $\mathbb S^N_K$ of constant sectional curvature $K >0$ and $\Omega \subset \mathbb S^N_K$ is convex with respect to the round geometry. Then we consider the conformal metric $g = e^{2\varphi}g_{\mathbb S^N_K}$. As before, we let 
\begin{align}\label{choice-of-sigma-max}
    \sigma_{\max } = \max\{ 0, \max \sigma ( \Hess _{g_{\mathbb S^N_K}}e^{2\varphi})\}
\end{align}
which is either $0$ if $e^{2\varphi}$ is concave or the largest eigenvalue of the hessian of $e^{2\varphi},$ otherwise.
\begin{thm}\label{thm: small-deformation-of-sphere}
    There exists $\varepsilon = \varepsilon(N,K)>0$ such that whenever $\sigma_{\max}< \varepsilon,$ one has that 
    \begin{align}
        \Gamma(\Omega) \geq \frac{(\min_\Omega \rho-\sigma_{\max}\tfrac{D^2}{8})^3}{\min_\Omega \rho(\max_\Omega \rho)^2}\frac{3\pi^2}{D^2}-\frac{(N-2)}{4(N-1)}\textup{osc}_\Omega (R_g)-\frac{N(N-2)}{4}\textup{osc}_\Omega (\exp(-2\varphi)).
    \end{align}
\end{thm}
\begin{proof}
    Let $u$ be the $g$-eigenfunction on $\Omega.$ Then, applying a conformal change procedure as before, we have that $ue^{\tfrac{N-2}{2}\varphi}$ satisfies the Schr\"odinger equation \eqref{eqn: Conformal eigenfunction equation, Schrodinger form}, where   $V = - \frac{N-2}{4(N-1)}e^{2\varphi}R_{\tilde g}+\frac{N-2}{4(N-1)}R_g.$ Thus, from the Raleigh quotient, we get that 
    \begin{align*}
        \Gamma(\Omega, \Delta_{g}) \geq \Gamma(\Omega, \Delta_{g_{\mathbb S^N_K}}, e^{2\varphi})-\frac{(N-2)}{4(N-1)}\textup{osc}_\Omega (R_g)-\frac{N(N-2)}{4}\textup{osc}_\Omega(\exp(-2\varphi)).
    \end{align*}
    Therefore, it suffices to estimate $\Gamma(\Omega, \Delta_{g_{\mathbb S^N_K}}, e^{2\varphi}).$ As before, we construct a modulus of concavity $\overline \rho $ given by 
    \begin{align*}
        \overline \rho(s) = \sigma_{\max} \tfrac{s^2}{2}+C,
    \end{align*}
    where $\sigma_{\max}$ is chosen as in the proof of Theorem \ref{thm: Gap-Estimate-for-conformally-flat-manifolds}. Note that in order to apply the comparison, we can choose $\varepsilon>0$ small enough such that for $C = \min _\Omega \rho-\tfrac{D^2}{8}\sigma_{\max}>0,$ $\overline \rho$ satisfies 
    \begin{align*}
        \overline \rho \leq \min \exp(2\varphi) \quad \& \quad \overline \rho' \leq 2\tn_K \overline \rho.
    \end{align*}
   To finish the proof, we need to verify that $\overline \lambda(\overline \rho) \leq \lambda_1(\Omega,  \rho),$ for $\rho = e^{2\varphi}.$ To do so, note that 
    \begin{align*}
        \lambda_1(\Omega, e^{2\varphi}) = \inf _{f \in H^1_0(\Omega) }\frac{\int_\Omega |\nabla f|^2\, dV_{g_{\mathbb S^N_K}}}{\int_\Omega \rho f^2\, dV_{g_{\mathbb S^N_K}}}\geq \frac{1}{\max _\Omega \rho}\lambda_1(\Omega) > \frac{1}{\max_ \Omega \rho}\frac{\pi^2}{D_{\mathbb S^N_K}^2},
    \end{align*}
    by Ling's estimate \cite{ling2006lower}, where $D_{\mathbb{S}^N_K}$ denotes the diameter with respect to $g_{\mathbb S^N_K}$. We now consider the problem \eqref{eqn: one-dimensional-model-weighted} on the interval $[-L/2,L/2]$ where we choose $ L = \tfrac{\sqrt{\max_\Omega}\rho}{\sqrt{\min \overline \rho}}D_{\mathbb S^N_K},$ so that 
    \begin{align*}
        \overline\lambda(\rho, L) \leq \frac{1}{\min \overline \rho}\frac{\pi^2}{L^2} \leq \lambda_1(\Omega, \rho). 
    \end{align*}
    We thus conclude from Proposition \ref{thm: super-log-concavity-general} that $(\log \overline \varphi_1)'+\tfrac{N-1}{2}\tn_K$ is a modulus of concavity. Thus by Theorem \ref{Gap-Comparison-Improved}, we conclude that 
    \begin{align*}
    \Gamma(\Omega, \Delta_{g_{\mathbb S^N_K}}, \rho) \geq \frac{\min \overline \rho}{\max_\Omega \rho}\Gamma(\overline \rho, L).
    \end{align*}
    In view of Theorem \ref{One-dimensional-fundamental-gap-theorem}, we conclude that 
    \begin{align*}
        \Gamma(\Omega, \Delta_{g_{\mathbb S^N_K}}, \rho) \geq \frac{(\min \overline \rho)^2}{(\max \overline \rho)^2}\frac{1}{\max_\Omega \rho}\frac{3\pi^2}{L^2} =  \frac{(\min \overline \rho)^3}{(\max \overline \rho)^2}\frac{1}{(\max_\Omega \rho)^2}\frac{3\pi^2}{D_{\mathbb S^N_K}^2}.
    \end{align*}
    Finally, since  $D_{\mathbb S^N_K} \leq  D/\min \exp(\varphi),$ we conclude 
    \begin{align*}
         \Gamma(\Omega, \Delta_{g_{\mathbb S^N_K}}, \rho) \geq  \frac{(\min_\Omega \rho-\sigma_{\max}\tfrac{D^2}{8})^3}{\min_\Omega \rho(\max_\Omega \rho)^2}\frac{3\pi^2}{D^2}.
    \end{align*}
\end{proof}

\bibliography{references}
\bibliographystyle{alpha}

\appendix

\section{Rectangles in the disk model of hyperbolic space}

\label{We need the inradius}

In this appendix, we construct domains in the Poincar\'e disk model of hyperbolic space which are convex with respect to this geometry but have arbitrarily small fundamental gap. This shows that our assumption that there is a lower bound on the inradius is necessary.

We define the domain $\Omega$ to be the Euclidean rectangle with vertices $(\pm L, \pm r)$ in the disk. Here, $L$ is some fixed value and $r$ will be very small.

For the most part, the proof is the same as the proof of Theorem 1 in \cite{khan2022negative}, with two differences.
\begin{enumerate}
    \item We can construct the domain to be symmetric, which eliminates the need for a continuity family at the end of the proof.
    \item In order to create the neck, we cannot use the growth of Jacobi fields since the domain is convex with respect to the Euclidean geometry, not the underlying hyperbolic geometry. However, we will make use of the fact that the conformal factor exceeds one near the left and right hand sides of the rectangle to replicate this effect.
\end{enumerate}

For $r$ small, we consider three smaller rectangles contained in $\Omega$.
\begin{eqnarray*}
    \Omega_{Left} &=& \{ (x,y) ~|~ -L<x<-L/2,  -r<y<r\} \\
    \Omega_{Neck} &=& \{ (x,y) ~|~ -L/4<x<L/4,  -r<y<r\} \\
    \Omega_{Right} &=& \{ (x,y) ~|~ L/2<x<L,  -r<y<r\} \\
\end{eqnarray*}

Now, we compute the heights of each of these rectangles with respect to the underlying hyperbolic metric. In other words, we consider the distance from the top of each rectangle to the bottom. Note that this will not be a constant value since the metric is non-Euclidean and will depend on the $x$-value. Then, a computation of the hyperbolic distances shows that for all points in $\Omega_{Left}$ and $\Omega_{Right}$, the height of the domain exceeds $\arcsinh\left( \frac{r}{4-L^2-4r^2}\right)$. Meanwhile, the height of points in $\Omega_{Neck}$ does not exceed
$\arcsinh\left(\frac{32r}{16-L^2-16r^2}\right)$. We then consider the ratio of these bounds to find the following:
\begin{equation}
   \textrm{height}(\Omega_{Left}) >  \frac{\arcsinh\left( \frac{r}{4-L^2-4r^2}\right)}{\arcsinh\left(\frac{32r}{16-L^2-16r^2}\right)} \textrm{height}(\Omega_{Neck}).
\end{equation}

Taking the limit as $r$ goes to zero, L'H\'opital's rule shows that
\begin{equation}
   \frac{\arcsinh\left( \frac{r}{4-L^2-4r^2}\right)}{\arcsinh\left(\frac{32r}{16-L^2-16r^2}\right)} \to \frac{ 16- L^2}{16-4L^2}.
\end{equation}
Therefore, we can find $r$ small enough so that
\begin{equation}
   \textrm{height}(\Omega_{Left}) >  \frac{ 16- L^2}{16-2L^2} \textrm{height}(\Omega_{Neck}).
\end{equation}
Fixing $L$ positive, $\frac{ 16- L^2}{16-2L^2} >1$, which suffices to repeat the proof from \cite{khan2022negative} verbatim. In particular, we can bound the eigenvalue of $\Omega$ from below by the eigenvalue of $\Omega_{Left}$, which is at most $\arcsinh^{-2}\left( \frac{r}{4-L^2-4r^2}\right) +l.o.t.$. On the other hand,  the eigenvalue of $\Omega_{Neck}$ exceeds $\arcsinh^{-2}\left(\frac{32r}{16-L^2-16r^2}\right)$, so the eigenvalue of the entire domain is much smaller (by a multiplicative factor) compared to the eigenvalue of the middle of the domain. Repeating the same doubling estimate from \cite{khan2022negative}, this forces the principle eigenfunction to become very small through $\Omega_{Neck}$ (roughly $O(\exp(-cr^{-1})$), and we can multiply the principle eigenfunction by a cut-off function to obtain a second function whose Rayleigh quotient is very similar to $\lambda_1$. Furthermore, since the domain and geometry are symmetric, so long as we choose this cut-off function to be odd, the resulting function is automatically orthogonal to the principle eigenfunction (eliminating the need for a continuity argument). The difference between the Rayleigh quotients of these two functions bounds the fundamental gap, so the fundamental gap of these rectangles gets arbitrarily small as the inradius (i.e. $r$) goes to zero.

\end{document}